\documentclass[11pt, oneside]{amsart}

\usepackage[british,english]{babel} 
\usepackage{graphics,color,pgf}
\usepackage{epsfig}
\usepackage[ansinew]{inputenc}
\usepackage[all]{xy}
\usepackage{hyperref}
\newdir{ >}{!/8pt/@{}*@{>}}
\usepackage{amssymb, amsmath,amsthm, mathtools, amscd}
\usepackage{mathrsfs}
\usepackage{stmaryrd}
\usepackage[margin=1.5in]{geometry}
\usepackage{enumerate}

\makeatletter
\newtheorem*{rep@theorem}{\rep@title}
\newcommand{\newreptheorem}[2]{%
\newenvironment{rep#1}[1]{%
 \def\rep@title{#2 \ref{##1}}%
 \begin{rep@theorem}}%
 {\end{rep@theorem}}}
\makeatother

\newtheorem{intro_thm}{Theorem}

\newtheorem{intro_prop}[intro_thm]{Proposition}

\theoremstyle{plain}
\newtheorem{teor}{Theorem}[section]
\newreptheorem{teor}{Theorem}  
\newtheorem{lem}[teor]{Lemma}

\newtheorem{prop}[teor]{Proposition}
\newreptheorem{prop}{Proposition}  

\newtheorem{setup}[teor]{Setup}

\theoremstyle{definition}
\newtheorem{deft}[teor]{Definition}
\newtheorem{es}[teor]{Example}

\theoremstyle{remark}
\newtheorem{oss}[teor]{Remark}
\newtheorem{notation}[teor]{Notation}

\DeclareMathOperator\sign{sign}

\DeclareMathOperator\vol{Vol}

\DeclareMathOperator\upC{\textup{C}}

\DeclareMathOperator\upH{\textup{H}}
\DeclareMathOperator\upI{\textup{I}}

\DeclareMathOperator\upL{\textup{L}}

\DeclareMathOperator\NCT{\mbox{NCT}}

\DeclareMathOperator\bbC{\mathbb{C}}

\DeclareMathOperator\bbH{\mathbb{H}}

\DeclareMathOperator\bbN{\mathbb{N}}

\DeclareMathOperator\bbR{\mathbb{R}}
\DeclareMathOperator\bbS{\mathbb{S}}

\DeclareMathOperator\calB{\mathcal{B}}

\DeclareMathOperator\calN{\mathcal{N}}

\DeclareMathOperator\calS{\mathcal{S}}

\DeclareMathOperator\calX{\mathcal{X}}

\DeclareMathOperator\pu{\textup{PU}}
\DeclareMathOperator\po{\textup{PO}}
\DeclareMathOperator\comp{comp}

\DeclareMathOperator\trans{trans}

\title[Multiplicative constants and maximal cocycles]{Multiplicative constants and maximal measurable cocycles in bounded cohomology}

\author[]{M. Moraschini}
\address{Fakult\"{a}t f\"{u}r Mathematik, Universit\"{a}t Regensburg, 93040 Regensburg, Germany}
\email{marco.moraschini@ur.de}

\author[]{A. Savini}
\address{Section de Math\'ematiques, University of Geneva, Rue Du Conseil-General, Geneva 1205, Switzerland}
\email{Alessio.Savini@Unige.ch}

\thanks{}

\subjclass[2010]{}
\date{\today.\ \copyright{\ M.~Moraschini, A. Savini 2019}.
  The first author was supported by the CRC~1085 \emph{Higher Invariants}
  (Universit\"at Regensburg, funded by the DFG). The second author was partially supported by the project \emph{Geometric and harmonic analysis with applications}, funded by EU Horizon 2020 under the Marie Curie grant agreement No 777822.}

\begin{document}

\begin{abstract}
Multiplicative constants are a fundamental tool in the study of maximal representations. In this paper we show how to extend such notion, and the associated framework, to measurable cocycles theory. As an application of this approach, we define and study the Cartan invariant for measurable $\textup{PU}(m,1)$-cocycles of complex hyperbolic lattices. 
\end{abstract}

\maketitle


\section{Introduction}

A fruitful approach to the study of geometric structures on a topological space $X$ is to introduce a bounded numerical invariant whose maximum detects those structures on $X$ which have many symmetries. An instance of this situation is the study of the representation space of lattices in (semi)simple Lie groups. More precisely, given two simple Lie groups of non-compact type $G,G'$ and a lattice $\Gamma \leq G$, Burger and Iozzi~\cite{BIuseful} described how to associate to every representation $\rho \colon \Gamma \rightarrow G'$ a real number. Using the pullback map $\upH^\bullet_{cb}(\rho)$ induced by $\rho$ in continuous bounded cohomology, they defined a numerical invariant $\lambda(\rho)$, which depends on a chosen class  $\Psi' \in \upH^\bullet_{cb}(G';\bbR)$, as follows: 
$$
\lambda(\rho) \coloneqq \langle \comp^\bullet_\Gamma \circ \upH^\bullet_{cb}(\rho)(\Psi'), [\Gamma \backslash \calX] \rangle \ ,
$$
where $\comp^\bullet_\Gamma$ denotes the comparison map (Section~\ref{subsec:transfer:maps}), $\calX$ denotes the Riemannian symmetric space associated to $G$, $[\Gamma \backslash \calX]$ is the (relative) fundamental class of the quotient manifold and $\langle \cdot, \cdot \rangle$ is the Kronecker product. We say that $\lambda(\rho)$ is a \emph{multiplicative constant} if it appears in an integral formula, called \emph{useful formula} by Burger and Iozzi~\cite{BIuseful}. When $\lambda$ is a multiplicative constant, the formula implies that the numerical invariant has bounded absolute value. 
In several cases~\cite{bucher2:articolo,Pozzetti,BBIborel}, its maximum corresponds precisely to representations induced by representations of the ambient group. 

\subsection{A multiplicative formula for measurable cocycles}

One of the main goal of this paper is to settle the foundational framework to define multiplicative constants for measurable cocycles. We carefully choose a setting where we can coherently extend ordinary numerical invariants for representations. Moreover, we introduce an integral formula in such a way that our definition of multiplicative constants is the natural extension of Burger-Iozzi's one. Our techniques make use of bounded cohomology theory.

Let $G,G'$ be two locally compact groups and let $L,Q \leq G$ be two closed subgroups. Assume that $Q$ is \emph{amenable} and that $L$ is a lattice. Let $(X,\mu_X)$ be a \emph{standard Borel probability $L$-space} and let $Y$ be a measurable $G'$-space. Following Burger-Iozzi's approach, given a measurable cocycle $\sigma \colon L \times X \rightarrow G'$, we define the \emph{pullback induced by $\sigma$} in continuous bounded cohomology using directly continuous cochains on the groups (Definition~\ref{def:pullback:cocycle}). 
Unfortunately, this approach does not lead to the desired multiplicative formula. 
For this reason, we need to consider boundary maps. 
A \emph{(generalized) boundary map} $\phi \colon G/Q \times X \rightarrow Y$ is a measurable $\sigma$-equivariant map and its existence is strictly related to the properties of $\sigma$ (Remark~\ref{oss:esistenza:mappe:bordo}). Inspired by the definition of Bader-Furman-Sauer's Euler number~\cite{sauer:articolo}, assuming the existence of a boundary map $\phi$, we describe how to construct a new \emph{pullback map} $\upC^\bullet(\Phi^X)$ in terms of $\phi$ (Definition \ref{def:pullback:not:fibered}). 
The notation $\upC^\bullet(\Phi^X)$ emphazises the fact that it is not simply the pullback along $\phi$, but we also need to integrate over $X$ (compare with Definitions~\ref{def:pullback:boundary} and~\ref{def:integration:map}).
The map induced by $\upC^\bullet(\Phi^X)$ in continuous bounded cohomology agrees with the natural pullback along $\sigma$ (Lemma \ref{lem:pullback:implemented:boundary:map}).

Our aim is to coherently extend the study of numerical invariants of representations to the case of measurable cocycles. 
Recall that given a continuous representation $\rho \colon L \rightarrow G'$ with boundary map $\varphi$, there always exists a natural \emph{measurable cocycle} $\sigma_\rho$ associated to it. Using the previous pullback $\upC^\bullet(\Phi^X)$, we then show that the map induced by $\rho$ in continuous bounded cohomology agrees with the one induced by $\sigma_\rho$ (Proposition~\ref{prop:pullback:coc:vs:repr}). Moreover, the pullback along $\sigma$ is invariant along the $G'$-cohomology class of the cocycle (Proposition~\ref{prop:invariance:cohomology}).

The study of pullback maps along measurable cocycles (and their boundary maps) leads to the following \emph{multiplicative formula}, which extends Burger-Iozzi's  useful formula~\cite[Proposition 2.44, Principle 3.1]{BIuseful}. Recall that the \emph{transfer map} is a cohomological left inverse of the restriction from $G$ to $L$.

\begin{intro_prop}[Multiplicative formula]\label{prop:baby:formula}
Keeping the notation above, let $\psi' \in \calB^\infty(Y^{\bullet+1};\bbR)^{G'}$ be an everywhere defined $G'$-invariant cocycle. Let $\psi \in \upL^\infty((G/Q)^{\bullet+1})^G$ be a $G$-invariant cocycle and let $\Psi \in \upH^\bullet_{cb}(G;\bbR)$ denote the class of $\psi$. Assume that $\Psi=\textup{trans}_{G/Q}^{\bullet} [\upC^\bullet(\Phi^X)(\psi')]$, where $\textup{trans}_{G/Q}^\bullet$ is the trasfer map. 
\begin{enumerate}
	\item We have that
$$
\int_{L \backslash G} \int_X \psi'(\phi(\overline{g}.\eta_1, x), \ldots, \phi(\overline{g}.\eta_{\bullet+1}, x)) d\mu_X(x) d\mu(\overline{g}) = \psi(\eta_1, \ldots, \eta_{\bullet+1}) +  \textup{cobound.} \ ,
$$
for almost every $(\eta_1,\ldots,\eta_{\bullet+1}) \in (G/Q)^{\bullet+1}$.
	\item If $\upH^\bullet_{cb}(G;\bbR) \cong \bbR \Psi  (= \bbR[\psi])$, then there exists a real constant $\lambda_{\psi',\psi}(\sigma) \in \bbR$ depending on $\sigma,\psi',\psi$ such that
\begin{align*}
\int_{L\backslash G} \int_X \psi'(\phi(\overline{g}.\eta_1, x), \ldots, \phi(\overline{g}.\eta_{\bullet+1}, x)) d\mu_X(x) d\mu(\overline{g})&=\lambda_{\psi',\psi}(\sigma) \cdot \psi(\eta_1,\ldots,\eta_{\bullet+1}) \\ &+\textup{cobound.} \ ,
\end{align*}
for almost every $(\eta_1,\ldots,\eta_{\bullet+1}) \in (G/Q)^{\bullet+1}$.
\end{enumerate}
\end{intro_prop} 

Although this formula might appear slightly complicated at first sight, it contains all the ingredients for defining the \emph{multiplicative constant} $\lambda_{\psi',\psi}(\sigma)$ associated to a measurable cocycle $\sigma$ and two given bounded cochains $\psi,\psi'$ (Definition~\ref{def:multiplicative:constant}). When no coboundary terms appear in the previous formula, we provide an explicit upper bound for the multiplicative constant (Proposition~\ref{prop:multiplicative:upperbound}).
This leads to the definition of \emph{maximal measurable cocycles} (Definition \ref{def:maximal:cocycle}). Finally, under suitable hypothesis, we prove that a maximal cocycle is \emph{trivializable} (Theorem~\ref{teor:coniugato:standard:embedding}), i.e. it is cohomologous to a cocycle induced by a representation $L \leq G \rightarrow G'$. 

This general framework has the great advantage that we can easily deduce several applications (Section~\ref{subsec:applications:baby} and Section~\ref{sec:concluding:remarks}).

\subsection{Cartan invariant of measurable cocycles}

Let $\Gamma \leq \pu(n,1)$ be a torsion-free lattice with $n \geq 2$. The study of representations of $\Gamma$ into $\pu(m,1)$ dates back to the work of Goldman and Millson~\cite{GoldM}, Corlette~\cite{Cor88} and Toledo~\cite{Toledo89}. In order to investigate rigidity properties of maximal representations $\rho \colon \Gamma \rightarrow \pu(m,1)$, Burger and Iozzi~\cite{BIcartan} defined the \emph{Cartan invariant} $i_\rho$ associated to $\rho$. Inspired by their work, we make use of  our techniques to define the \emph{Cartan invariant} $i(\sigma)$ for a measurable cocycle $\sigma:\Gamma \times X \rightarrow \pu(m,1)$, where $(X,\mu_X)$ is a standard Borel probability $\Gamma$-space.

If the cocycle admits a boundary map (e.g. if it is \emph{non elementary}), the Cartan invariant can be realized as the multiplicative constant associated to $\sigma$ and the Cartan cocycles $c_n,c_m$. More precisely, as an application of Proposition~\ref{prop:baby:formula}, we prove the following

\begin{intro_prop}\label{prop:cartan:multiplicative:cochains}
Let $\Gamma \leq \pu(n,1)$ be a torsion-free lattice and let $(X,\mu_X)$ be a standard Borel probability space. Consider a non-elementary measurable cocycle $\sigma \colon \Gamma \times X \rightarrow \pu(m,1)$ with boundary map $\phi \colon \partial_\infty \bbH^n_{\bbC} \times X \rightarrow \partial_\infty \bbH^m_{\bbC}$.
Then, for every triple of pairwise distinct points $\xi_1,\xi_2,\xi_3 \in \partial_\infty \bbH^n_{\bbC}$, we have
\begin{small}
\begin{equation*}
i(\sigma)c_n(\xi_1,\xi_2,\xi_3)=\int_{\Gamma \backslash \pu(n,1)} \int_X c_m(\phi(\overline{g}\xi_1,x),\phi(\overline{g}\xi_2,x),\phi(\overline{g}\xi_3,x))d\mu(\overline{g})d\mu_X(x) \ .
\end{equation*}
\end{small}
Here $\mu$ is a $\pu(n,1)$-invariant probability measure on the quotient $\Gamma \backslash \pu(n,1)$. 
\end{intro_prop}

First we show that our Cartan invariant extends the one defined for representations (Proposition~\ref{prop:cartan:rep}). Moreover, using our results about the pullback along boundary maps, we prove that the Cartan invariant is constant along $\pu(m,1)$-cohomology classes and it has absolute value bounded by $1$ (Proposition \ref{prop:cartan:cohomology:bound}).

Then, a natural problem is to provide a complete characterization of measurable cocycles whose Cartan invariant attains extremal values, i.e. either $0$ or $1$. Since we are not interested in elementary cocycles, we can assume the existence of a boundary map~\cite[Proposition 3.3]{MonShal0}. 

Following the work by Burger and Iozzi~\cite{BIreal}, we introduce the notion of \emph{totally real cocycles}. A cocycle is totally real if it is cohomologous to a cocycle whose image is contained in a subgroup of $\pu(m,1)$ preserving a totally geodesically embedded copy $\bbH^k_{\bbR} \subset \bbH^m_{\bbR}$, for some $1 \leq k \leq m$ (Definition \ref{def:totally:real}). Totally real cocycles can be easily constructed by taking the composition of a measure equivalence cocycle with a totally real representation. 

We show that totally real cocycles have trivial Cartan invariant. The converse seems unlikely to hold in general. However, if $X$ is $\Gamma$-ergodic, we obtain the following 

\begin{intro_thm}\label{teor:totally:real}
Let $\Gamma \leq \pu(n,1)$ be a torsion-free lattice and and let $(X,\mu_X)$ be a standard Borel probability $\Gamma$-space. 
Consider a non-elementary measurable cocycle $\sigma \colon \Gamma \times X \rightarrow \pu(m,1)$ with boundary map $\phi \colon \partial_\infty \mathbb{H}^n_{\bbC} \times X \rightarrow \partial_\infty \bbH^m_{\bbC}$. Then the following hold
\begin{enumerate}
	\item  If $\sigma$ is totally real, then $i(\sigma)=0$;
	\item  If $X$ is $\Gamma$-ergodic and $\textup{H}^2(\phi)([c_m])=0$, then $\sigma$ is totally real.
\end{enumerate}
\end{intro_thm}

The next step in our investigation is the study of the \emph{algebraic hull} of a cocycle with non-vanishing pullback. Recall that the algebraic hull is the smallest algebraic group containing the image a cocycle cohomologous to $\sigma$ (Definition \ref{def:alg:hull}). 

\begin{intro_thm}\label{teor:alg:hull}
Let $\Gamma \leq \textup{PU}(n,1)$ be a torsion-free lattice and let $(X,\mu_X)$ be an ergodic standard Borel probability $\Gamma$-space. Consider a non-elementary measurable cocycle $\sigma \colon \Gamma \times X \rightarrow \pu(m,1)$ with boundary map $\phi \colon \partial_\infty \bbH^n_{\bbC} \times X \rightarrow \partial_\infty \bbH^m_{\bbC}$.  Let $\mathbf{L}$ be the algebraic hull of $\sigma$ and denote by $L=\mathbf{L}(\bbR)^\circ$ the connected component of the identity of the real points. 

If $\upH^2(\Phi^X)([c_m])\neq 0$, then $L$ is an almost direct product $K \cdot M$, where $K$ is compact and $M$ is isomorphic to $\textup{PU}(p,1)$ for some $1 \leq p \leq m$. 

In particular, the symmetric space associated to $L$ is a totally geodesically embedded copy of $\bbH^p_{\bbC}$ inside $\bbH^m_{\bbC}$. 
\end{intro_thm}

We conclude with a complete characterization of maximal cocycles.

\begin{intro_thm}\label{teor:cartan:rigidity}
Consider $n \geq 2$. Let $\Gamma \leq \pu(n,1)$ be a torsion-free lattice. Let $(X,\mu_X)$ be an ergodic standard Borel probability $\Gamma$-space. Consider a maximal measurable cocycle $\sigma \colon \Gamma \times X \rightarrow \pu(m,1)$. 
Let $\mathbf{L}$ be the algebraic hull of $\sigma$ and let $L=\mathbf{L}(\bbR)^\circ$ be the connected component of the identity of the real points. 

Then, we have
\begin{enumerate}
\item $m \geq n$; 

\item $L$ is an almost direct product $\textup{PU}(n,1) \cdot K$, where $K$ is compact;

\item $\sigma$ is cohomologous to the cocycle $\sigma_i$ associated to the standard lattice embedding $i \colon \Gamma \to \pu(m,1)$ (possibly modulo the compact subgroup $K$ when $m >n$). 
\end{enumerate}
\end{intro_thm}

Since recently one of the authors together with Sarti proved a generalization of the previous theorem for cocycles with target $\textup{PU}(p,q)$~\cite[Theorem 2]{sarti:savini}, we will mainly refer to their more complete result for the proof.

\subsection*{Plan of the paper}

In Section~\ref{sec:prel}, we recall some preliminary definitions and results that we need in the paper. We report  the definitions of amenable action, measurable cocycle, boundary map and algebraic hull in Section~\ref{subsec:zimmer:cocycle}. We then review Burger and Monod's functorial approach to continuous bounded cohomology (Section~\ref{subsec:burger:monod}) and we conclude this preliminary section with the definition of transfer maps (Section~\ref{subsec:transfer:maps}).

Section \ref{sec:easy:formula} is devoted to the description of the general framework in which we study multiplicative constants associated to measurable cocycles. There, we first define the pullback along a measurable cocycle and along its boundary map (Section~\ref{subsec:pullback:boundary}). Then, we compare our definition with the usual one given for representations (Section~\ref{subsec:rep:coc}). In Section \ref{subsec:easy:mult:formula} we state our multiplicative formula (Proposition~\ref{prop:baby:formula}) and we introduce the notion of multiplicative constant associated to a measurable cocycle. We conclude the section studying the notion of maximality (Section~\ref{subsec:multiplicative:constant}) and showing some applications of the previous results (Section~\ref{subsec:applications:baby}).

Section~\ref{sec:cartan:invariant} contains the new application of our machinery. There, we introduce and study the Cartan invariant of measurable cocycles (Section~\ref{sec:cartan:invariant}). We prove that it is a multiplicative constant (Proposition \ref{prop:cartan:multiplicative:cochains}) and it extends the same invariant for representations (Proposition \ref{prop:cartan:rep}). Moreover, we show that the Cartan invariant has bounded absolute value in Proposition~\ref{prop:cartan:cohomology:bound}.

In Section \ref{sec:tot:real} we define totally real cocycles and we prove Theorem \ref{teor:totally:real}. Then in Section \ref{sec:cartan:rigidity} we study maximal measurable cocycles and we prove both Theorem \ref{teor:alg:hull} and Theorem \ref{teor:cartan:rigidity}. We conclude with some remarks about recent applications of our theory in Section \ref{sec:concluding:remarks}.

\subsection*{Acknowlegdements}

We truly thank the anonymous referee for the detailed report which allowed to substantially improve the quality of our paper. 

\section{Preliminary definitions and results}\label{sec:prel}

\subsection{Amenability and measurable cocycles} \label{subsec:zimmer:cocycle} 
In this section we are going to recall some classic definitions related to both amenability and measurable cocycles.
We start fixing the following notation:
\begin{itemize}
\item Let $G$ be a locally compact second countable group endowed with its Haar measurable structure.

\item Let $(X,\mu)$ be a \emph{standard Borel measure $G$-space}, i.e. a standard Borel measure space endowed with a measure-preserving $G$-action.
\end{itemize}
If $\mu$ is a probability measure, we will refer to $(X, \mu)$ as a \emph{standard Borel probability $G$-space}. Given another measure space $(Y,\nu)$, we denote by $\textup{Meas}(X,Y)$ the space of measurable functions from $X$ to $Y$ endowed with the topology of convergence in measure.
\begin{oss}
In the literature about the ergodic version of simplicial volume~\cite{Sauer, Sthesis, FFM,LP,FLPS,FLF,Camp-Corro,FLMQ},
it is often convenient to work with essentially free actions. For this reason, one might find reasonable to stick with the same assumptions also here working with the dual notion of bounded cohomology. However, it is easy to check that every probability measure-preserving action can be promoted to an essentially-free action just by taking the product with an essentially free action and considering the diagonal action on that product. 
\end{oss}

We recall now some definitions about amenability. We mainly refer the reader to the books by Zimmer~\cite[Section 4.3]{zimmer:libro} and by Monod~\cite[Section~5.3]{monod:libro} for further details about this topic.

Let $\upL^\infty(G;\bbR)$ denote the space of essentially bounded real functions over $G$. Then, $G$ acts on $\upL^\infty(G;\bbR)$ as follows
$$
g.f (g_0) = f(g^{-1} g_0) \ ,
$$
for all $g, g_0 \in \, G$ and $f \in \, \upL^\infty(G;\bbR)$.

\begin{deft}\label{def:amenable:group}
A \emph{mean} on $\upL^\infty(G;\bbR)$ is a continuous linear functional
$$
m \colon \upL^\infty(G;\bbR) \rightarrow \bbR \ ,
$$
such that $m(f)\geq 0$ whenever $f \geq 0$ and $m(\chi_G)=1$, where $\chi_G$ denotes the characteristic function on $G$. 

We say that a mean is \emph{left invariant} if for all $g \in \, G$ and $f \in \, \upL^\infty(G;\bbR)$ we have $m(g.f) = m(f)$.

A group is \emph{amenable} if it admits a left-invariant mean.
\end{deft}

\begin{es}\label{es:amenable:groups}
The following families are examples of amenable groups:
\begin{enumerate}
\item Abelian groups~\cite[Theorem 4.1.2]{zimmer:libro};

\item Finite/Compact groups~\cite[Theorem 4.1.5]{zimmer:libro};

\item Extensions of amenable groups by amenable groups~\cite[Proposition 4.1.6]{zimmer:libro};


\item Let $G$ be a Lie group and let $P \subset G$ be any minimal parabolic subgroup. Then, $P$ is an extension of a solvable group by a compact group, whence it is amenable~\cite[Corollary 4.1.7]{zimmer:libro}.
\end{enumerate}
\end{es}

In the sequel we will need a more general notion of amenability which is related to group actions. In fact, amenable spaces and amenable actions will play a crucial role in the functorial approach to the computation of continuous bounded cohomology (Section \ref{subsec:burger:monod}). Following Monod's convention, we begin by defining \emph{regular} $G$-{spaces}~\cite[Definition~2.1.1]{monod:libro}.

\begin{deft}
Let $G$ be a locally compact second countable group and let $S$ be a standard Borel space with a measurable $G$-action which preserve a measure class. We say that $(S, \mu)$ is a \emph{regular} $G$-\emph{space} if the previous measure class contains a probability measure $\mu$ such that the isometric action $R \colon G \curvearrowright \upL^1(S, \mu)$ defined by
$$
(R(g).f)(s) =  f(g^{-1}.s) \frac{d g^{-1}\mu}{d\mu}(s)
$$
is continuous. Here $d g^{-1}\mu \slash d\mu$ denotes the Radon-Nikod\'{y}m derivative.
\end{deft}

\begin{es}\label{es:reg:spaces}
Let $G$ be locally compact second countable group, then the following are examples of regular $G$-spaces~\cite[Example~2.1.2]{monod:libro}:
\begin{enumerate}
\item If $G$ is endowed with its Haar measure, then $G$ is a regular $G$-space. 

\item If $Q$ is a closed subgroup of $G$, then $G \slash Q$ endowed with the natural almost invariant measure is a regular $G$-space.

\item Furstenberg-Poisson boundaries~\cite{furst:articolo73,furst:articolo} associated to a probability measure on $G$ are regular $G$-spaces.
\end{enumerate}
\end{es}

The notion of regular $G$-spaces allows us to introduce the definitions of \emph{amenable actions} and \emph{amenable spaces}~\cite[Theorem~5.3.2]{monod:libro}.

\begin{deft}\label{def:amenable:action}
Let $G$ be a locally compact second countable group and let $(S,\mu)$ be a regular $G$-space. We say that the action of $G$ on $(S,\mu)$ is \emph{amenable} if  there exists a continuous norm-one $G$-equivariant linear operator
$$
p \colon \upL^\infty(G \times S;\bbR) \rightarrow \upL^\infty(S;\bbR) \ , 
$$
with the following two properties: First $p(\chi_{G \times S})=\chi_S$, secondly for all $f \in \, \upL^\infty(G \times S)$ and for all measurable sets $A \subset S$ we have $p(f \cdot \chi_{G \times A}) = p(f) \cdot \chi_A$.

If the action by $G$ on $(S, \mu)$ is amenable, then we say that $(S,\mu)$ is an \emph{amenable $G$-space}.
\end{deft}

\begin{oss}\label{oss:am:spaces}
The previous definition extends the notion of amenable groups in the following sense: A group is amenable if and only if every regular $G$-space is an amenable $G$-space~\cite[Theorem~5.3.9]{monod:libro}.

Amenable actions not only characterize groups but also subgroups. By Example~\ref{es:reg:spaces}.2, given a closed subgroup $Q \subset G$, the quotient $G \slash Q$ is a regular $G$-space. Additionally, we have that $Q$ is amenable if and only if the $G$-action on $G \slash Q$ is amenable~\cite[Proposition 4.3.2]{zimmer:libro}. Hence, Example~\ref{es:amenable:groups}.4 shows that if $G$ is a Lie group and $P \subset G$ is any minimal parabolic subgroup, then the $G$-action on the quotient $G \curvearrowright G \slash P$ is amenable. This applies to the Furstenberg-Poisson boundary of a Lie group (being identified with $G/P$). 
\end{oss}

We recall now the notion of measurable cocycles and some of their properties.

\begin{deft}\label{def:zimmer:cocycle}
Let $G$ and $H$ be locally compact groups and let $(X,\mu)$ be a standard Borel probability $G$-space. A \emph{measurable cocycle} (or, simply \emph{cocycle}) is a measurable map $\sigma \colon G \times X \rightarrow H$ satisfying the following formula
\begin{equation}\label{eq:zimmer:cocycle}
\sigma(g_1g_2,x)=\sigma(g_1,g_2.x)\sigma(g_2,x) \ ,
\end{equation}
for almost every $g_1,g_2 \in G$ and almost every $x \in X$. Here, $g_2.x$ denotes the action by $g_2 \in G$ on $x \in \, X$. 
\end{deft}

Associated to measurable cocycles there exists the crucial notion of \emph{boundary map}.

\begin{deft}
Let $G$ and $H$ be two locally compact groups and let $Q \leq G$ be a closed amenable subgroup. Let $(X,\mu)$ be a standard Borel probability $G$-space and let $(Y,\nu)$ be a measure space on which $H$ acts by preserving the measure class of $\nu$. Given a measurable cocycle $\sigma \colon G \times X \rightarrow H$, we say that a measurable map $\phi \colon G/Q \times X \rightarrow Y$ is \emph{$\sigma$-equivariant} if we have
$$
\phi(g.\eta,g.x)=\sigma(g,x)\phi(\eta,x) \ ,
$$
for almost every $g \in G,\eta \in G/Q$ and $x \in X$. 

A \emph{(generalized) boundary map} associated to $\sigma$ is a $\sigma$-equivariant measurable map. 
\end{deft}

We will make use of generalized boundary maps in Section~\ref{subsec:pullback:boundary}, when we will explain how to compute the pullback in continuous bounded cohomology.

\begin{oss}\label{oss:esistenza:mappe:bordo}
It is quite natural to ask when a (generalized) boundary map actually exists. Let $G(n)=\textup{Isom}(\bbH^n_{K})$ be the isometry group of the $K$-hyperbolic space, where $K$ is either $\bbR$ or $\bbC$. Given a lattice $\Gamma \leq G(n)$, let us consider a standard Borel probability $\Gamma$-space $(X,\mu_X)$ and a measurable cocycle $\sigma:\Gamma \times X \rightarrow G(m)$. In the previous situation, Monod and Shalom \cite[Proposition 3.3]{MonShal0} proved that if the cocycle $\sigma$ is \emph{non elementary} then there exists an essentially unique boundary map 
$$
\phi:\partial_\infty \bbH^n_{K} \times X \rightarrow \partial_\infty \bbH^m_{K} \ .
$$
The notion of non-elementary cocycle relies on the definition of \emph{algebraic hull} (Definition \ref{def:alg:hull}) and it will be explained more carefully later in this paper. 

Also in the case of higher rank lattices there are some relevants results about the existence of boundary maps. Indeed a key step in the proof of Zimmer' Superrigidity Theorem \cite[Theorem 4.1]{zimmer:annals} is to prove the existence of generalized boundary maps for \emph{Zariski dense} measurable cocycles (see Definition \ref{def:alg:hull}). 
\end{oss}

Since Equation~(\ref{eq:zimmer:cocycle}) suggests that $\sigma$ can be interpreted as a Borel $1$-cocycle in $\textup{Meas}(G,\textup{Meas}(X,H))$~\cite{feldman:moore}, it is natural to introduce the definition of \emph{cohomologous cocycles}. 

\begin{deft}\label{def:cohomology:cocycle}
Let $\sigma \colon G \times X \rightarrow H$ be a measurable cocycle between locally compact groups. Let $f \colon X \rightarrow H$ be a measurable map. We define the \emph{twisted cocycle associated to $\sigma$ and $f$} as
$$
f.\sigma \colon G \times X \rightarrow H, \hspace{5pt} (f.\sigma)(g,x) \coloneqq f(g.x)^{-1}\sigma(g,x)f(x) \ ,
$$ 
for almost every $g \in G$ and almost every $x \in X$. 

We say that two measurable cocycles $\sigma_1,\sigma_2\colon G \times X \rightarrow H$ are \emph{cohomologous} if there exists a measurable function $f \colon X \rightarrow H$ such that
$$
\sigma_2=f.\sigma_1 \ .
$$
Similarly, we say that $\sigma_1$ and $\sigma_2$ are \emph{cohomologous modulo a closed subgroup $C \leq H$} if
$$
\sigma_2=f.\sigma_1 \ \ \mod C \ ,
$$
that is
$$
\sigma_2(g,x) \cdot (f.\sigma_1(g,x))^{-1} \in C \ ,
$$
for almost every $g \in G, x \in X$. 
\end{deft}

When a measurable cocycle $\sigma$ admits a generalized boundary map, then all its cohomologous cocycles share the same property. 

\begin{deft}\label{def:twisted:map}
Let $\sigma \colon G \times X \rightarrow H$ be a measurable cocycle with generalized boundary map $\phi \colon G/Q \times X \rightarrow Y$. Given a measurable function $f \colon X \rightarrow H$ the \emph{twisted boundary map associated to $f$ and $\phi$} is defined as
$$
f.\phi \colon G/Q \times X \rightarrow Y, \hspace{5pt} (f.\phi)(\eta,x) \coloneqq f(x)^{-1}\phi(\eta,x) \ ,
$$
for almost every $g \in G,\eta \in G/Q$ and $x \in X$. 
\end{deft}
\begin{oss}\label{oss:twisted:boundary:map}
Let $\sigma, \sigma' \colon G \times X \rightarrow H$ be two cohomologous cocycles and let $f \colon X \rightarrow H$ be the measurable map such that $\sigma' = f.\sigma$. If $\sigma$ admits a generalized boundary map $\phi$, then the twisted boundary map associated to $f$ and $\phi$ is a generalized boundary map associated to $\sigma'$.
\end{oss}

Representations provide special cases of measurable cocycles: 

\begin{deft}\label{def:rep:cocycle}
Let $\rho \colon G \rightarrow H$ be a continuous representation and let $(X,\mu)$ be a standard Borel probability $G$-space. The \emph{cocycle associated to the representation $\rho$} is defined as
$$
\sigma_\rho \colon G \times X \rightarrow H, \hspace{5pt} \sigma_\rho(g,x)=\rho(g) \ , 
$$
for every $g \in G$ and almost every $x \in X$. 
\end{deft}

Given a representation $\rho \colon G \rightarrow H$, one can obtain useful information about $\rho$ by studying the closure of the image $\overline{\rho(\Gamma)}$ as subgroup of $H$. 
On the other hand, in general the image of a measurable cocycle does not have any nice algebraic structure. 
Nevertheless, when $H$ is assumed to be an algebraic group, we can give the following

\begin{deft}\label{def:alg:hull}
Let $\mathbf{H}$ be a real algebraic group and denote by $H=\mathbf{H}(\bbR)$ the set of real points of $\mathbf{H}$. The \emph{algebraic hull} of a measurable cocycle $\sigma\colon G \times X \rightarrow H$ is the (conjugacy class of the) smallest algebraic subgroup $\mathbf{L} \leq \mathbf{H}$ such that $\mathbf{L}(\bbR)^\circ$ contains the image of a cocycle cohomologous to $\sigma$. Here $\mathbf{L}(\bbR)^\circ$ denotes the connected component of the neutral element.
\end{deft}

\begin{oss}
The algebraic hull is well-defined by the Noetherian property of algebraic groups. Moreover, it only depends on the cohomology class of the cocycle $\sigma$~\cite[Proposition 9.2.1]{zimmer:libro}. 
\end{oss}
We will use the previous definition when we will work with~\emph{totally real cocycles} (Section~\ref{sec:tot:real}) and when we will investigate the properties of cocycles with non-vanishing pullback (Theorem~\ref{teor:alg:hull}). 

\subsection{Bounded cohomology and its functorial approach} \label{subsec:burger:monod}

In this section we are going to recall the definitions and the properties of both continuous and continuous bounded cohomology that we will need in the sequel.

We first introduce continuous (bounded) cohomology via the homogeneous resolution and then, following the work by Burger and Monod \cite{monod:libro,burger2:articolo}, we describe it in terms of strong resolutions by relatively injective modules.

\begin{deft}
Let $G$ be a locally compact  group. A \emph{Banach} $G$-\emph{module} $(E, \pi)$ is a Banach space $E$ endowed with a $G$-action induced by a representation $\pi \colon G \to \textup{Isom}(E)$, or equivalently a $G$-action via linear isometries:
\begin{align*}
\theta_\pi \colon G \times E &\to E \\
\theta_\pi (g, v) &\coloneqq \pi (g) v \ .
\end{align*}
We say that a Banach $G$-module $(E, \pi)$ is \emph{continuous} if the map $\theta(\cdot, v)$ is continuous for all $v \in \, E$. 
Finally, we denote by $E^G$ the submodule of $G$-invariant vectors in $E$, i.e. the space of vectors $v$ such that $\theta(g, v) = v$ for all $g \in \, G$.
\end{deft}
\begin{notation}
In the sequel $\mathbb{R}$ will denote the Banach $G$-module of trivial real coefficients. In other words, it is endowed with the trivial $G$-action: $\pi(g) v = v$ for all $v \in \, \mathbb{R}$ and $g \in \, G$.
\end{notation}

\begin{es}
Let $(E, \pi)$ be a Banach $G$-module. Then, the space of \emph{continuous} $E$-\emph{valued functions}
$$
\textup{C}_{c}^\bullet(G; E) \coloneqq \{f \colon G^{\bullet+1} \rightarrow E \, | \,  \mbox{$f$ is continuous} \} 
$$
is a continuous Banach $G$-module with the following action
\begin{equation}\label{eq:azione:funzioni}
g.f (h_1, \cdots, h_{\bullet +1}) = \pi(g) f(g^{-1} h_1, \cdots, g^{-1} h_{\bullet+1}) \ ,
\end{equation}
for all $g, h_1, \cdots, h_{\bullet+1} \in \, G$.
\end{es}

The spaces of continuous $E$-valued functions give raise to a cochain complex $(\textup{C}_{c}^\bullet(G; E), \delta^\bullet)$ together with the standard homogeneous coboundary operator
$$
\delta^\bullet \colon \textup{C}_{c}^\bullet(G; E) \rightarrow \textup{C}_{c}^{\bullet + 1}(G; E)
$$
$$
\delta^\bullet (f) (g_1, \ldots, g_{\bullet + 2}) \coloneqq \sum_{j = 1}^{\bullet + 2} (-1)^{j - 1} f(g_1, \ldots, g_{j-1}, g_{j+1} \ldots, g_{\bullet + 2}) \ .
$$
Since this complex is exact, we are going to focus our attention to the subcomplex of $G$-invariant vectors $(\upC_c^\bullet(G; E)^G, \delta^\bullet)$.

\begin{deft}
The \emph{continuous cohomology} of $G$ with coefficients in $E$, denoted by $\upH_{c}^\bullet(G; E)$, is the cohomology of the complex $(\upC_{c}^\bullet(G; E)^G, \delta^\bullet)$.
\end{deft}
\begin{oss}\label{oss:discr:group:cohom:same:G}
If $G$ is a discrete group, then there is no difference between continuous and ordinary cohomology. 
Hence, in this situation we will usually drop the subscript $c$ from the notation.
\end{oss}

Since $(E, \lVert \cdot \rVert_E)$ is a Banach space, the Banach $G$-module $\upC_c^\bullet(G; E)$ has a natural $\upL^\infty$-norm: For every $f \in \, \upC_c^\bullet(G; E)$, we have
$$
\lVert f \rVert_\infty \coloneqq \sup \{ \lVert f(g_1, \ldots, g_{\bullet + 1} )\rVert_E \, | \, g_1, \ldots, g_{\bullet +1} \in \, G    \} \ .
$$
A continuous function is said to be  \emph{bounded} if its $\upL^\infty$-norm is finite. Let $\upC_{cb}^\bullet(G; E) \subset \upC_c^\bullet(G; E)$ be the subspace of continuous bounded functions. 
By linearity the coboundary operator $\delta^\bullet$ preserves boundedness, so we can restrict $\delta^\bullet$ to the space of continuous bounded $G$-invariant functions $\upC_{cb}^\bullet(G;E)^G$. Then we get the following complex
$$
(\upC_{cb}^\bullet(G; E)^G, \delta^\bullet) \ .
$$ 

\begin{deft}
The \emph{continuous bounded cohomology} of $G$ with coefficients in $E$, denoted by $\upH_{cb}^\bullet(G; E)$,  is the cohomology of the complex $(\upC_{cb}^\bullet(G; E)^G, \delta^\bullet)$.
\end{deft}
\begin{oss}
If $L \subset G$ is a closed subgroup, then we can compute the continuous bounded cohomology of $L$ with $E$-coefficients as the cohomology of the complex $(\upC_{cb}^\bullet(G;E)^L, \delta^\bullet)$. Here the $L$-action is the restriction of the natural $G$-action on $\upC_{cb}^\bullet(G;E)$~\cite[Corollary~7.4.10]{monod:libro}.
\end{oss}

The $\upL^\infty$-norm defined on cochains induces a canonical $\upL^\infty$-seminorm in cohomology given by
$$
\lVert f \rVert_\infty  \coloneqq \inf \{ \lVert \psi \rVert_\infty \ | \ [\psi]=f \}  \ . 
$$
We say that an isomorphism between seminormed cohomology groups is \emph{isometric} if  the corresponding seminorms are preserved.

Beyond the difference determined by the quotient seminorm, one can study the gap between continuous cohomology and continuous bounded cohomology via the map induced in cohomology  by the inclusion
$
i \colon \upC_{cb}^\bullet(G; E)^G \rightarrow \upC^\bullet_{c}(G; E)^G$.
The resulting map $$\comp_G^\bullet \colon \upH_{cb}^\bullet(G; E) \rightarrow \upH_c^\bullet(G; E) $$ is called \emph{comparison map}.

In the sequel we will need an alternative description of continuous bounded cohomology in terms of strong resolutions via relatively injective modules. Since we will not make an explicit use of these notions, we refer the reader to Monod's book for a broad discussion on them~\cite[Section~4.1 and~7.1]{monod:libro}. The main result in this direction is the following
\begin{teor}[{\cite[Theorem~7.2.1]{monod:libro}}]\label{thm:Monod:strong:rel:inj}
Let $G$ be a locally compact  group and let $(E, \pi)$ be a Banach $G$-module. Then, for every strong resolution $(E^\bullet, \delta^\bullet)$ of $E$ via relatively injective $G$-modules, the cohomology of the complex of $G$-invariants $\upH^n((E^\bullet)^G, \delta^\bullet))$ is isomorphic as a topological vector space to the continuous bounded cohomology $\upH^n_{cb}(G; E)$, for every $n \geq 0$.
\end{teor}

We now describe a strong resolution via relatively injective modules which allows us to compute bounded cohomology \emph{isometrically}. Let $G$ be a locally compact second countable group. Let $(E, \pi, \lVert \cdot \rVert_E)$ be a Banach $G$-module such that $E$ is the dual of some Banach space. This implies that $E$ can be endowed with the weak-$^*$ topology and the associated weak-$^*$ Borel structure. Moreover, let $(S, \mu)$ be a regular $G$-space. We have the following

\begin{deft}
We define the Banach $G$-module of \emph{bounded weak}-$^*$ \emph{measurable} $E$-\emph{valued functions on} $S$ to be the Banach space
\begin{align*}
\mathcal{B}^\infty(S^{\bullet+1};E) \coloneqq \{ f \colon S^{\bullet+1} \rightarrow E \, | \,  &\textup{$f$ is weak-$^*$ measurable}, \\
\sup_{s_1,\cdots,s_{\bullet+1} \in S} \lVert &f(s_1,\cdots,s_{\bullet+1}) \rVert_E < \infty \} \ 
\end{align*}
endowed with the following $G$-action $\tau$:
$$
(\tau(g) f) (s_1, \cdots, s_{\bullet+1}) \coloneqq \pi(g) f(g^{-1}s_1, \cdots, g^{-1}s_{\bullet+1})
$$
for every $g \in \, G$, $s_1, \cdots, s_{\bullet+1} \in \, S$ and $f \in \, \mathcal{B}^\infty(S^{\bullet+1};E)$.

We define the Banach $G$-module of \emph{essentially bounded weak-$^*$ measurable $E$-valued functions} on $S$ to be
$$
\textup{L}^\infty_{\textup{w}^*}(S^{\bullet+1};E) \coloneqq \{ [f]_\sim \ | \ f \in \, \mathcal{B}^\infty(S^{\bullet+1};E)\} \ ,
$$
where $f \sim g$ if and only if they agree $\mu$-almost everywhere and $[f]_\sim$ denotes the equivalence class of $f$ with respect to $\sim$. 
\end{deft}
\begin{oss}
For ease of notation we will denote elements in $\textup{L}^\infty_{\textup{w}^*}(S^{\bullet+1};E)$ simply by one chosen representative $f$.
\end{oss}
\begin{deft}\label{def:alternating}
Let us consider the situation above. We say that a(n essentially) bounded weak-~$^*$ measurable function $f \colon S^{\bullet + 1} \rightarrow E$ is \emph{alternating} if for every $s_1, \cdots, s_{\bullet+1} \in \, S$ we have
$$
\sign(\varepsilon) f(s_1, \ldots, s_{\bullet + 1}) =  f(s_{\varepsilon(1)}, \ldots, s_{\varepsilon({\bullet + 1})}) \ ,
$$
where $\varepsilon \in \mathfrak{S}_{\bullet+1}$ is a permutation whose sign is $\sign(\varepsilon)$. 
\end{deft}
Since the standard homogeneous operator $\delta^\bullet$ preserves $G$-invariant (alternating) essentially bounded weak-${}^*$ measurable functions up to a shift of the degree, we can consider the complex $(\textup{L}^\infty_{\textup{w}^*}(S^{\bullet+1};E), \delta^\bullet)$. The following theorem shows when the previous complex computes isometrically the continuous bounded cohomology of $G$ with $E$-coefficients.
\begin{teor}\cite[Theorem 7.5.3]{monod:libro}\label{teor:monod:2:rel:inj:strong}
Let $G$ be a locally compact second countable group. Let $(E, \pi)$ be a dual Banach $G$-module. Let $(S, \mu)$ be an \emph{amenable} regular $G$-space. Then, the cohomology of the complex $(\textup{L}^\infty_{\textup{w}^*}(S^{\bullet+1};E)^G, \delta^\bullet)$ is \emph{isometrically} isomorphic to $\upH^n_{cb}(G; E)$, for every integer $n \geq 0$.

The same result still holds if we restrict to the subcomplex of alternating essentially bounded weak-${}^*$ measurable functions on $S$.
\end{teor}
\begin{oss}\label{oss:L:resol:uguale:G}
In the situation of the previous theorem if $L \subset G$ is a closed subgroup, then also the cohomology of the complex $(\textup{L}^\infty_{\textup{w}^*}(S^{\bullet+1};E)^L, \delta^\bullet)$ is \emph{isometrically} isomorphic to $\upH^n_{cb}(L; E)$, for every $n \geq 0$~\cite[Lemma~4.5.3]{monod:libro}.
\end{oss}
\begin{es}\label{es:L:inf:G:Q}
 Let $G$ be a locally compact second countable group and let $Q \subset G$ be a closed amenable subgroup. By Remark~\ref{oss:am:spaces} and Example~\ref{es:reg:spaces}.2 we know that $G \slash Q$ is an amenable regular $G$-space. Thus for every Banach $G$-module $(E, \pi)$ the cohomology of the complex $(\textup{L}^\infty_{\textup{w}^*}((G \slash Q)^{\bullet+1};E)^G, \delta^\bullet)$ isometrically computes the continuous bounded cohomology of $G$ with coefficient in $E$. An instance of this situation is when $Q$ is a minimal parabolic subgroup of a semisimple Lie group $G$. 
\end{es}

As we have just discussed one can compute continuous bounded cohomology by working with equivalence classes of bounded weak-${}^*$ measurable functions. However, in some cases it might be convenient to work directly with $\mathcal{B}^\infty(S^{\bullet + 1}; E)$. 
Also in this case the  homogeneous coboundary operator sends (alternating) bounded weak-$^*$ measurable functions to themselves up to shifting the degree. Hence, we can still construct a complex 
$(\mathcal{B}^\infty(S^{\bullet + 1}; E), \delta^\bullet)$. Unfortunately, the associated resolution of $E$ is only strong in general~\cite[Proposition~2.1]{burger:articolo}.
So it cannot be used to compute the continuous bounded cohomology of $G$ with $E$-coefficients.
Nevertheless, one obtain the following canonical map~\cite[Corollary~2.2]{burger:articolo} 
\begin{equation}\label{eq:canonical:map:B:L}
\mathfrak{c}^n \colon \upH^{n}(\mathcal{B}^\infty(S^{\bullet + 1}; E)^G) \rightarrow\upH^{n}(\upL_{\text{w}^*}^\infty(S^{\bullet + 1}; E)^G) \cong \upH_{cb}^n(G; E) \ ,
\end{equation}
for every $n \in \, \mathbb{N}$. This shows that each  bounded weak-${}^*$ measurable $G$-invariant function canonically determines a cohomology class in $\upH_{cb}^n(G; E)$. The same result still holds in the situation of alternating functions.

In Section~\ref{subsec:pullback:boundary}, we will tacitly use this result for showing that the pullback of a bounded weak-${}^*$ measurable $G$-invariant function lies in fact in $\upL^\infty_{\text{w}^*}$.

\subsection{Transfer maps}\label{subsec:transfer:maps}

In this section we briefly recall the notion of \emph{transfer maps}~\cite{monod:libro}. Let $G$ be a locally compact second countable group and let $i \colon L \rightarrow G$ be the inclusion of a closed subgroup $L$ into $G$. By functoriality the inclusion induces a pullback in continuous bounded cohomology
$$
\upH^\bullet_{cb}(i) \colon \upH_{cb}^\bullet(G; \mathbb{R}) \rightarrow \upH_{cb}^\bullet(L; \mathbb{R}) \ .
$$

\begin{oss}\label{oss:notation:restriction}
Since the map $\upH^\bullet_{cb}(i)$ is implemented by the restriction to $L$ of cochains on $G$, we will sometimes write $\kappa|_L$ instead of $\upH^\bullet_{cb}(i)(\kappa)$, for $\kappa \in \upH^\bullet_{cb}(G;\bbR)$.  
\end{oss}

A transfer map provides a cohomological left inverse to $\upH^\bullet_{cb}(i) $. Assume that $L \backslash G$ admits a $G$-invariant probability measure $\mu$ (e.g. when $L$ is a lattice of $G$), then we have

\begin{deft}\label{def:trans:map}
We define the \emph{transfer cochain map} as
$$
\widehat{\textup{trans}}^\bullet_L \colon \upC_{cb}^\bullet(G;\bbR)^L \rightarrow \upC_{cb}^\bullet(G;\bbR)^G
$$
$$
\widehat{\textup{trans}}^\bullet_L(\psi)(g_1, \ldots, g_{\bullet+1}) \coloneqq \int_{L \backslash G} \psi(\overline{g}.g_1, \ldots, \overline{g}.g_{\bullet+1}) d\mu(\overline{g}) \ ,
$$
for every $(g_1, \ldots, g_{\bullet+1}) \in \, G^{\bullet+1}$ and $\psi \in \, \upC_{cb}^\bullet(G;\bbR)^L$. Here $\overline{g}$ denotes the equivalence class of $g$ in the quotient $L \backslash G$. 

The \emph{transfer map} $\trans_L^\bullet$ is the one induced in cohomology by $\widehat{\textup{trans}}^\bullet_L$:
$$
\trans_L^\bullet \colon \upH_{cb}^\bullet(L;\bbR) \rightarrow \upH_{cb}^\bullet(G;\bbR) \ .
$$
\end{deft} 

\begin{oss}
The transfer map is well-defined since we can compute the continuous bounded cohomology of $L$ by looking at the complex $(\upC_{cb}^\bullet(G;\bbR)^L, \delta^\bullet)$ (Remark~\ref{oss:discr:group:cohom:same:G}).

Moreover, since $\psi$ is $L$-invariant, it induces a well-defined function on the quotient $L \backslash G$. 
With a slight abuse of notation, in the previous formula we still denoted by $\psi$ this induced function.
\end{oss}

We give now an alternative definition of the transfer map for essentially bounded weak-${}^*$ measurable functions. 
Let $Q$ and $L$ be closed subgroups of a locally compact second countable group $G$. If $Q$ is amenable, 
then the subcomplex of $L$-invariant essentially bounded functions on $G/Q$ computes the continuous bounded cohomology $\upH^\bullet_{cb}(L;\bbR)$
(Remark~\ref{oss:L:resol:uguale:G} and Example~\ref{es:L:inf:G:Q}). Hence, the new \emph{transfer map}
$$\trans_{G \slash Q}^\bullet \colon \upH^\bullet_{cb}(L; \mathbb{R}) \rightarrow \upH^\bullet_{cb}(G; \mathbb{R})$$ 
is the map induced in cohomology by the following
$$
\widehat{\textup{trans}}^\bullet_{G \slash Q} \colon \upL^\infty((G \slash Q)^{\bullet+1};\bbR)^L \rightarrow  \upL^\infty((G \slash Q)^{\bullet+1};\bbR)^G
$$
$$
\widehat{\textup{trans}}^\bullet_{G \slash Q}(\psi)(\xi_1, \ldots, \xi_{\bullet+1}) \coloneqq \int_{L \backslash G} \psi(\overline{g}.\xi_1, \ldots, \overline{g}.\xi_{\bullet+1}) d \mu(\overline{g}) \ ,
$$
for almost all $(\xi_1, \ldots, \xi_{\bullet+1}) \in \, (G \slash Q)^{\bullet + 1}$ and $\psi \in \, \upL^\infty((G \slash Q)^{\bullet+1};\bbR)^L$.

The following commutative diagram completely describes the relation between the two transfer maps $\trans^\bullet_L$ and $\trans^\bullet_{G \slash Q}$~\cite[Lemma 2.43]{BIuseful} 
\begin{equation}\label{eq:diagramma:due:transf}
\xymatrix{
\upH^\bullet_{cb}(L; \mathbb{R}) \ar[rr]^-{\trans_L^\bullet} \ar[d]_-\cong && \upH^\bullet_{cb}(G; \mathbb{R}) \ar[d]^-{\cong} \\
\upH^\bullet_{cb}(L; \mathbb{R}) \ar[rr]_-{\trans^\bullet_{G \slash Q}} && \upH^\bullet_{cb}(G; \mathbb{R})  \ .
}
\end{equation}
Here the vertical arrows are the canonical isomorphisms obtained by extending the identity $\mathbb{R} \rightarrow \mathbb{R}$ to the complex of continuous bounded and essentially bounded functions, respectively.

\section{Pullback maps, multiplicative constants and maximal measurable cocycles}\label{sec:easy:formula}

The main goal of this section is to define pullbacks in continuous bounded cohomology via measurable cocycles and generalized boundary maps. As an application we extend Burger and Iozzi's \emph{useful formula} for representations~\cite[Proposition~2.44]{BIuseful} to the wider setting of measurable cocycles. This allows us to introduce the notion of multiplicative constants and investigate cocycles rigidity.

\begin{setup}\label{setup:mult:const}
Let us consider the following setting:
\begin{itemize}
	\item Let $G$ be a second countable locally compact group. 
	\item Let $G'$ be a locally compact group which acts measurably on a measure space $(Y, \nu)$ by preserving the measure class.
	\item Let $Q$ be a closed amenable subgroup of $G$. 
	\item Let $L$ be a lattice in $G$.
	\item Let $(X,\mu_X)$ be a standard Borel probability $L$-space.
	\item  Let $\sigma \colon L \times X \rightarrow G'$ be a measurable cocycle with an essentially unique generalized boundary map $\phi \colon G/Q \times X \rightarrow Y$.
\end{itemize}
\end{setup}

\subsection{Pullback along measurable cocycles and generalized boundary maps}\label{subsec:pullback:boundary}

In this section we introduce two different pullback maps in continuous bounded cohomology associated to a measurable cocycle. The first pullback will only depend on the measurable cocycle $\sigma$. The second one will be defined in terms of the generalized boundary map $\phi$. We will show that under suitable conditions the two definitions agree (Lemma~\ref{lem:pullback:implemented:boundary:map}). Despite a priori the first definition might appear more natural, we will mainly exploit the second pullback in the study of the rigidity properties of measurable cocycles.

Given a measurable cocycle $\sigma \colon L \times X \rightarrow G'$ we define a pullback map from $\upC^\bullet_{cb}(G';\bbR)^{G'}$ to $\upC^\bullet_b(L;\bbR)^L$ as follows (compare with~\cite[Remark 14]{moraschini:savini}).

\begin{deft}\label{def:pullback:cocycle}
In the situation of Setup \ref{setup:mult:const}, the \emph{pullback map induced by the measurable cocycle $\sigma$} is given by
$$
\upC_b^\bullet(\sigma) \colon \upC^\bullet_{cb}(G';\bbR) \rightarrow \upC^\bullet_b(L;\bbR) \ , 
$$
$$
\psi \mapsto \upC^\bullet_b(\sigma)(\psi)(\gamma_1,\ldots,\gamma_{\bullet+1}) \coloneqq \int_X(\psi(\sigma(\gamma_1^{-1},x)^{-1}),\ldots,\sigma(\gamma_{\bullet+1}^{-1},x)^{-1})d\mu_X(x) \ .
$$
\end{deft}

\begin{oss}
The previous formula takes inspiration both from Bader-Furman-Sauer's result~\cite[Theorem 5.6]{sauer:companion} and Monod-Shalom's cohomological induction for measurable cocycles associated to couplings~\cite[Section 4.2]{MonShal}.
\end{oss}

\begin{lem}\label{lem:pullback:map:cocycle:cohomology}
In the situation of Setup~\ref{setup:mult:const}, the map $\upC^\bullet_b(\sigma)$ is a well-defined cochain map which restricts to the subcomplexes of invariant cochains. Hence $\upC^\bullet_b(\sigma)$ induces a map in bounded cohomology
$$
\upH^\bullet_b(\sigma):\upH^\bullet_{cb}(G';\bbR) \rightarrow \upH^\bullet_b(L;\bbR)\ , \ \upH^\bullet_b(\sigma)([\psi]):=\left[ \upC^\bullet_b(\sigma)(\psi) \right] \ .
$$
\end{lem}

\begin{proof}
It is easy to check that $\upC^\bullet_b(\sigma)$ is a cochain map. Moreover, it sends bounded cochains to bounded cochains because $\mu_X$ is a probability measure.

It only remains to prove that $\upC^\bullet_b(\sigma)$ sends $G'$-invariant continuous cochains to $L$-invariant ones. Let $\psi \in \upC^\bullet_{cb}(G';\bbR)^{G'}$ and $\gamma,\gamma_1,\ldots,\gamma_{\bullet+1} \in L$, then we have
\begin{align*}
\gamma \cdot \upC^\bullet_b(\sigma)(\psi)(\gamma_1,\ldots,\gamma_{\bullet+1})&=\upC^\bullet_b(\sigma)(\psi)(\gamma^{-1}\gamma_1,\ldots,\gamma^{-1}\gamma_{\bullet+1})=\\
&=\int_X \psi(\sigma(\gamma^{-1}_1 \gamma,x)^{-1},\ldots,\sigma(\gamma^{-1}_{\bullet+1} \gamma,x)^{-1})d\mu_X(x)=\\
&=\int_X \psi(\sigma(\gamma,x)^{-1}\sigma(\gamma^{-1}_1,\gamma.x)^{-1},\ldots,\sigma(\gamma,x)^{-1}\sigma(\gamma_{\bullet+1}^{-1},\gamma.x)^{-1})d\mu_X(x)=\\
&=\int_X \psi(\sigma(\gamma,x)^{-1}\sigma(\gamma^{-1}_1,x)^{-1},\ldots,\sigma(\gamma,x)^{-1}\sigma(\gamma_{\bullet+1}^{-1},x)^{-1})d\mu_X(x)=\\
&=\int_X\psi(\sigma(\gamma^{-1}_1,x)^{-1},\ldots,\sigma(\gamma_{\bullet+1}^{-1},x)^{-1})d\mu_X(x)=\\
&=\upC^\bullet_b(\sigma)(\psi)(\gamma_1,\ldots,\gamma_{\bullet+1}) \ ,
\end{align*}
where the second line is equal to the third one because of the definition of measurable cocycle (Equation~\eqref{eq:zimmer:cocycle}). Then, the $L$-invariance of the measure $\mu_X$ shows that the third line is equal to the fourth one. Finally, the $G'$-invariance of $\psi$ concludes the computation.
\end{proof}

As anticipated we now explain how to define a different pullback map via generalized boundary maps in the situation of Setup~\ref{setup:mult:const}. This approach takes inspiration from a work by Bader, Furman and Sauer~\cite[Proposition~4.2]{sauer:articolo} and has already produced some applications in special settings (Subsection~\ref{subsec:applications:baby}).
We define the pullback along a generalized boundary map as the composition of two different maps defined in continuous bounded cohomology. The Banach space $\textup{L}^\infty(X) \coloneqq \textup{L}^\infty(X; \mathbb{R})$ has a natural structure of Banach $L$-module given by the following $L$-action
$$
\gamma.f = f(\gamma^{-1}.x) \ , 
$$ 
for all $\gamma \in L$ and $f \in \, \upL^\infty(X)$. This leads to the following

\begin{deft}\label{def:pullback:boundary}
In the situation of Setup~\ref{setup:mult:const}, the \emph{$\upL^\infty(X)$-pullback along $\phi$} is the following map
$$
\upC^\bullet(\phi) \colon \mathcal{B}^\infty(Y^{\bullet + 1}; \mathbb{R})^{G'} \rightarrow \textup{L}_{\text{w}^*}^\infty((G \slash Q)^{\bullet+1}; \upL^\infty(X))^L 
$$
$$
\upC^\bullet(\phi)(\psi) (\eta_1, \ldots, \eta_{\bullet+1}) \coloneqq \left(x \mapsto \psi(\phi(\eta_1, x), \ldots, \phi(\eta_{\bullet+1}, x))\right) \ ,
$$
where $\psi \in \, \mathcal{B}^\infty(Y^{\bullet + 1}; \mathbb{R})^{G'}$, $\eta_1, \ldots, \eta_{\bullet+1} \in \, G \slash Q$ and $x \in \, X$.
\end{deft}

\begin{lem}\label{lemma:pullback:cochain}
The map $\upC^\bullet(\phi)$ is a well-defined norm non-increasing cochain map. 
\end{lem}
\begin{proof}
Since $\upC^\bullet(\phi)$ is defined as a pullback, it is immediate to check that it is a norm non-increasing cochain map. 

Let us show now that for every $\psi \in \, \mathcal{B}^\infty(Y^{\bullet + 1}; \mathbb{R})^{G'}$, the cocycle $\upC^\bullet(\phi)(\psi)$ is $L$-invariant. First, by \cite[Corollary 2.3.3]{monod:libro} we can identify 
$$
\textup{L}_{\text{w}^*}^\infty((G \slash Q)^{\bullet+1}; \upL^\infty(X))^L \cong \textup{L}^\infty((G \slash Q)^{\bullet+1} \times X)^L \ ,
$$
where the latter space is endowed with its natural diagonal $L$-action. Then, for almost every $x \in \, X$, $\gamma \in \, L$ and $\eta_1, \ldots, \eta_{\bullet+1} \in \, G \slash Q$, we have
\begin{align*}
\gamma \cdot \upC^\bullet(\phi)(\psi)(\eta_1, \ldots, \eta_{\bullet+1}) (x) &= \upC^\bullet(\phi)(\psi)(\gamma^{-1}.\eta_1,\ldots,\gamma^{-1}. \eta_{\bullet+1})(\gamma^{-1}.x)=\\
&=\psi(\phi(\gamma^{-1}.\eta_1,  \gamma^{-1}.x), \ldots, \phi(\gamma^{-1}. \eta_{\bullet+1},\gamma^{-1}.x))= \\
&= \psi(\sigma(\gamma^{-1}, x) \phi(\eta_1, x), \ldots, \sigma(\gamma^{-1}, x) \phi(\eta_{\bullet+1}, x))=\\
&= \psi(\phi(\eta_1, x), \ldots, \phi(\eta_{\bullet+1}, x)) \\
&= \upC^\bullet(\phi)(\psi)(\eta_1, \ldots, \eta_{\bullet+1}) (x) \ .
\end{align*}
Here we first used the definition of diagonal action, then the $\sigma$-equivariance of $\phi$ and finally the $G'$-invariance of $\psi$.
\end{proof}

Since our final goal is to pullback a cocycle $\psi \in \, \mathcal{B}^\infty(Y^{\bullet + 1}; \mathbb{R})^{G'}$ along $\phi$ obtaining a new cocycle in $\textup{L}^\infty((G \slash Q)^{\bullet+1}; \mathbb{R})^L$, we need to compose the $\upL^\infty(X)$-pullback along $\phi$ with the integration map (compare with \cite{sauer:articolo, savini3:articolo, moraschini:savini}).

\begin{deft}\label{def:integration:map}
In the situation of Setup~\ref{setup:mult:const}, the \emph{integration map} $\upI_X^\bullet$ is the following cochain map
$$
\upI_X^\bullet \colon \textup{L}_{\text{w}^*}^\infty((G \slash Q)^{\bullet+1}; \upL^\infty(X))^L \rightarrow \textup{L}^\infty((G \slash Q)^{\bullet+1}; \mathbb{R})^L \\ 
$$
$$
\upI_X^\bullet(\psi)(\eta_1,\ldots,\eta_{\bullet+1}) \coloneqq \int_X \psi(\eta_1, \ldots, \eta_{\bullet+1})(x) d\mu_X(x) \ ,
$$
where $\psi \in \, \textup{L}^\infty((G \slash Q)^{\bullet+1}; \upL^\infty(X))^L$, $\eta_1, \ldots, \eta_{\bullet+1} \in \, G \slash Q$ and $\mu_X$ is the probabilty measure on the standard Borel probability $L$-space $X$.
\end{deft}

\begin{lem}\label{lemma:int:cochain:map}
The integration map $\upI^\bullet_X$ is a well-defined norm non-increasing cochain map.
\end{lem}
\begin{proof}
Given a cocycle $\psi \in \, \textup{L}_{\text{w}^*}^\infty((G \slash Q)^{\bullet+1}; \upL^\infty(X))^L$, it is easy to show that $\upI^\bullet_X(\psi)$ is $L$-invariant. Indeed, given $\eta_1,\ldots,\eta_{\bullet+1} \in G/Q$ and $\gamma \in \, L$, we have
\begin{align*}
\gamma . \upI_X^\bullet(\psi)(\eta_1, \ldots, \eta_{\bullet+1}) &= \int_X \psi(\gamma^{-1}.\eta_1, \ldots, \gamma^{-1}. \eta_{\bullet+1})(x) d\mu_X(x)=\\
&= \int_X \psi(\eta_1, \ldots, \eta_{\bullet+1})(\gamma.x) d\mu_X(x) \\
&= \int_X \psi(\eta_1, \ldots, \eta_{\bullet+1})(x) d\mu_X(x) = \upI^\bullet_X(\psi)(\eta_1, \ldots, \eta_{\bullet+1}) \ ,
\end{align*}
where we used the $L$-invariance of both $\psi$ and $\mu_X$.

Since it is immediate to check that the integration map is also a norm non-increasing cochain map, we get the thesis.
\end{proof}

\begin{oss}\label{oss:unbounded:cochains}
The previous construction via integration is only possible working with bounded cocycles. Indeed, there is no hope to extend this map to the case of unbounded cochains~\cite[Remark~13, Remark~16]{moraschini:savini}.
\end{oss}

We are now ready to define the \emph{pullback map along} $\phi$.

\begin{deft}\label{def:pullback:not:fibered}
In the situation of Setup~\ref{setup:mult:const}, the \emph{pullback map along the (generalized) boundary map} $\phi$ is the following cochain map
$$
\upC^\bullet(\Phi^X) \colon \mathcal{B}^\infty(Y^{\bullet + 1}; \mathbb{R})^{G'} \rightarrow \textup{L}^\infty((G \slash Q)^{\bullet+1}; \mathbb{R})^L
$$
$$
\upC^\bullet(\Phi^X) \coloneqq \upI^\bullet_X \circ \upC^\bullet(\phi) \ .
$$
\end{deft}

\begin{oss}\label{oss:pullback:restringe:alternating:cochains}
The restriction of the pullback along $\phi$ to the subcomplexes of alternating cochains (Definition~\ref{def:alternating}) is well-defined. 
\end{oss}

The fact that the pullback map induces a well-defined map in cohomology is proved in the following

\begin{prop}\label{prop:pullback:cohomology}
In the situation of Setup~\ref{setup:mult:const} the pullback map $\upC^\bullet(\Phi^X)$ is a norm non-increasing cochain map.
Hence, it induces a well-defined map
$$
\upH^\bullet(\Phi^X) \colon \upH^\bullet(\calB^\infty(Y^{\bullet+1};\bbR)^{G'}) \rightarrow \upH^\bullet_{b}(L;\bbR), \hspace{5pt} \upH^\bullet(\Phi^X)([\psi])\coloneqq[\upC^\bullet(\Phi^X)(\psi)] \ .
$$
The same result still holds for the subcomplexes of alternating cochains. 
\end{prop}

\begin{proof}
As a consequence of both Lemmas~\ref{lemma:pullback:cochain} and~\ref{lemma:int:cochain:map}, the pullback $\upC^\bullet(\Phi^X)$ is a norm non-increasing cochain map. Indeed, it is the composition of two such maps, namely $\upC^\bullet(\phi)$ and $\upI^\bullet_X$.

Since $Q$ is an amenable group, then $G\slash Q$ is an amenable regular $G$-space (Example~\ref{es:reg:spaces}.2 and Remark~\ref{oss:am:spaces}). Hence, by Remark~\ref{oss:L:resol:uguale:G} the complex of $L$-invariant essentially bounded functions $\upL^\infty((G/Q)^{\bullet+1};\bbR)^L$ computes the continuous bounded cohomology $\upH^\bullet_{b}(L;\bbR)$. 

The same proof adapts mutatis mutandis to the case of alternating cochains.
\end{proof}

\begin{oss} \label{oss:more:general:amenable}
One might define a pullback map in cohomology using any measurable $\sigma$-equivariant map $\phi \colon S \times X \rightarrow Y$, where $S$ is any amenable $L$-space.  However, since we will not need this formulation in the sequel, we preferred to keep the previous setting.
\end{oss}

Since we have introduced two different pullback maps in continuous bounded cohomology arising from measurable cocycles, it is natural to ask whether they agree. The following lemma completely describes the situation (compare with~\cite[Corollary 2.7]{burger:articolo}). 

\begin{lem}\label{lem:pullback:implemented:boundary:map}
In the situation of Setup~\ref{setup:mult:const}, let $\psi \in \calB^\infty(Y^{\bullet+1};\bbR)^{G'}$  be a measurable cocycle.
Then 
$$
\upC^\bullet(\Phi^X)(\psi) \in \upL^\infty((G/Q)^{\bullet+1};\bbR)^L 
$$
is a representative of the class $\upH^\bullet_b(\sigma)([\psi]) \in \upH^\bullet_{b}(L;\bbR)$.
\end{lem}

\begin{proof}
It is sufficient to consider the following commutative diagram~\cite[Proposition 1.2]{burger:articolo}
$$
\xymatrix{
\upH^\bullet(\calB^\infty(Y^{\bullet+1};\bbR)^{G'}) \ar[rrr]^-{\upH^\bullet(\Phi^X)} \ar[dd]_-{\mathfrak{c}^\bullet} &&& \upH^\bullet_b(L;\bbR)\\ 
\\
\upH^\bullet_{cb}(G';\bbR) \ar[uurrr]_-{\upH^\bullet_b(\sigma)} \ ,
}
$$
where $\mathfrak{c}^\bullet$ is the map introduced in Equation~(\ref{eq:canonical:map:B:L}).
\end{proof}

Finally, we show that the pullback along cohomologous measurable cocyles is the same (compare with~\cite[Proposition~13, Proposition~20]{moraschini:savini}).

\begin{prop}\label{prop:invariance:cohomology}
In the situation of Setup~\ref{setup:mult:const}, let $f.\sigma \colon L \times X \rightarrow G'$ be a cocycle cohomologous to $\sigma$ with respect to a measurable map $f \colon X \rightarrow G'$. Then, for every $\psi \in \mathcal{B}^\infty(Y^{\bullet + 1}; \mathbb{R})^{G'}$, we have $$\upC^\bullet(\Phi^X)(\psi) = \upC^\bullet(f.\Phi^X)(\psi) \ .$$
Here $\upC^\bullet(\Phi^X)$ and $\upC^\bullet(f.\Phi^X)$ denote the pullback maps along the associated boundary maps $\phi$ and $f.\phi$, respectively.
\end{prop}
\begin{proof}
The boundary map $f.\phi$ associated to $f.\sigma$ is given by
$$
f.\phi \colon G \slash Q \times X \rightarrow Y \ , \hspace{10pt} (f.\phi)(\eta, x) = f^{-1}(x) \phi(\eta, x) \ ,
$$
for almost every $\eta \in \, G \slash Q$ and $x \in \, X$ (Remark~\ref{oss:twisted:boundary:map}). Hence, we have
\begin{align*}
\upC^\bullet(f.\Phi^X)(\psi)(\eta_1, \ldots, \eta_{\bullet+1}) &= \int_X \psi((f.\phi)(\eta_1, x), \ldots, (f.\phi)(\eta_{\bullet +1}, x)) d\mu_X(x)= \\
&= \int_X \psi(f^{-1}(x) \phi(\eta_1, x), \ldots, f^{-1}(x)  \phi(\eta_{\bullet +1}, x)) d\mu_X(x)= \\
&= \int_X \psi( \phi(\eta_1, x), \ldots, \phi(\eta_{\bullet +1}, x)) d\mu_X(x)= \\
&=\upC^\bullet(\Phi^X)(\psi)(\eta_1, \ldots, \eta_{\bullet+1}) \ ,
\end{align*}
for almost every $\eta_1, \ldots, \eta_{\bullet +1 } \in \, G \slash Q$. This finishes the proof.
\end{proof}

\begin{oss}
Sometimes it is natural to consider the $G'$-module $\mathbb{R}$ with a twisted action. For instance if $G'$ admits a sign homomorphism, we can use it to twist the real coefficients. In that situation the previous equality will be true only up to a sign (see for instance~\cite[Proposition~13]{moraschini:savini}).
\end{oss}

\subsection{Pullback along  generalized boundary maps vs.~pullback of representations} \label{subsec:rep:coc}

In the situation of Setup~\ref{setup:mult:const}, let $(X,\mu_X)$ be a standard Borel probability $L$-space and let $\rho \colon L \rightarrow G'$ be a representation. 
Then, there exists an associated measurable cocycle $\sigma_\rho \colon L \times X \rightarrow G'$ defined by $\sigma_\rho(\gamma, x) = \rho(\gamma)$ for every $\gamma \in \, L$ and $x \in \, X$ (Definition~\ref{def:rep:cocycle}). 
If $\rho$ admits a $\rho$-equivariant measurable map $\varphi \colon G/Q \rightarrow Y$, 
the corresponding generalized boundary map of $\sigma_\rho$ is
$$
\phi \colon G/Q \times X \rightarrow Y, \hspace{5pt} \phi(\eta, x) = \varphi(\eta) \ ,
$$
for almost every $\eta \in \, G \slash Q$ and $x \in \, X$.

As explained by Burger and Iozzi~\cite{burger:articolo, BIuseful}, one can implement the pullback map 
$$\upH^\bullet_{cb}(\rho) \colon \upH^\bullet_{cb}(G'; \mathbb{R}) \rightarrow \upH^\bullet_{b}(L; \mathbb{R})$$ 
using a cochain map $\upC^\bullet(\varphi)$ defined by
$$
\upC^\bullet(\varphi):\calB^\infty(Y^{\bullet+1};\bbR)^{G'} \rightarrow \upL^\infty((G/Q)^{\bullet+1};\bbR)^L \ ,
$$
$$
\psi \mapsto \upC^\bullet(\varphi)(\psi)(\eta_1,\ldots,\eta_{\bullet+1}):=\psi(\varphi(\eta_1),\ldots,\varphi(\eta_{\bullet+1})) \ ,
$$
for almost every $\eta_1,\ldots,\eta_{\bullet+1} \in G/Q$. 

The following result shows that the pullback associated to $\rho$ via $\varphi$ agrees with the one along $\phi$. This property turns out to be fundamental to coherently extend the numerical invariants of representations to the ones of measurable cocycles (see~\cite[Proposition 3.4]{savini3:articolo} and~\cite[Propositions~12, Proposition~19]{moraschini:savini}).
\begin{prop}\label{prop:pullback:coc:vs:repr}
In the situation of Setup~\ref{setup:mult:const}, let $\rho \colon L \rightarrow G'$ be a representation which admits a $\rho$-equivariant measurable map $\varphi \colon G/Q \rightarrow Y$. Then, we have
$$
\upC^\bullet(\Phi^X)=\upC^\bullet(\varphi) \ .
$$
\end{prop}
\begin{proof}
Since the boundary map $\phi$ associated to $\sigma_\rho$ does not depend on the second variable, it is immediate to check that the following diagram commutes
$$
\xymatrix{
\mathcal{B}^\infty(Y^{\bullet + 1}; \mathbb{R})^{G'} \ar[rr]^-{\upC^\bullet(\phi)} \ar[rd]_-{\upC^\bullet(\varphi)} && \textup{L}_{\text{w}^*}^\infty((G \slash Q)^{\bullet+1}; \upL^\infty(X))^L  \ar[ld]^-{\upI_X^\bullet} \\
& \textup{L}^\infty((G \slash Q)^{\bullet+1}; \mathbb{R})^L \ , } 
$$
whence the thesis.
\end{proof}

\begin{oss}
The existence of a cocycle of the form $\sigma \colon L \times X \rightarrow G'$ required in Setup~\ref{setup:mult:const} is irrelevant in the previous result.
\end{oss}

\subsection{Multiplicative formula}\label{subsec:easy:mult:formula}

In this section we prove how to deduce the multiplicative formula stated in Proposition~\ref{prop:baby:formula}. Some applications of the formula are then discussed in  Section~\ref{subsec:applications:baby}.

\begin{repprop}{prop:baby:formula}
In the situation of Setup~\ref{setup:mult:const}, let $\psi' \in \calB^\infty(Y^{\bullet+1};\bbR)^{G'}$ be an everywhere-defined $G'$-invariant cocycle. Let $\psi \in \upL^\infty((G/Q)^{\bullet+1})^G$ be a $G$-invariant cocycle. Denote by $\Psi \in \upH^\bullet_{cb}(G;\bbR)$ the class of $\psi$. Assume that $\Psi=\textup{trans}_{G/Q}^{\bullet} [\upC^\bullet(\Phi^X)(\psi')]$. 
\begin{enumerate}
	\item We have that
$$
\int_{L \backslash G} \int_X \psi'(\phi(\overline{g}.\eta_1, x), \ldots, \phi(\overline{g}.\eta_{\bullet+1}, x)) d\mu_X(x) d\mu(\overline{g}) = \psi(\eta_1, \ldots, \eta_{\bullet+1}) +  \textup{cobound.} \ ,
$$
for almost every $(\eta_1,\ldots,\eta_{\bullet+1}) \in (G/Q)^{\bullet+1}$.
	\item If $\upH^\bullet_{cb}(G;\bbR) \cong \bbR \Psi  (= \bbR[\psi])$, then there exists a real constant $\lambda_{\psi',\psi}(\sigma) \in \bbR$ depending on $\sigma,\psi',\psi$ such that
\begin{align*}
\int_{L\backslash G} \int_X \psi'(\phi(\overline{g}.\eta_1, x), \ldots, \phi(\overline{g}.\eta_{\bullet+1}, x)) d\mu_X(x) d\mu(\overline{g})&=\lambda_{\psi',\psi}(\sigma) \cdot \psi(\eta_1,\ldots,\eta_{\bullet+1}) \\ &+\textup{cobound.} \ ,
\end{align*}
for almost every $(\eta_1,\ldots,\eta_{\bullet+1}) \in (G/Q)^{\bullet+1}$.
\end{enumerate}
\end{repprop} 
\begin{proof}
\emph{Ad~1.} Since Setup~\ref{setup:mult:const} ensures the existence of the transfer map $\trans_{G \slash Q}^\bullet$, the first formula is easily true.

\emph{Ad~2.} Since $\upH^\bullet_{cb}(G;\bbR)$ is one-dimensional and generated by $\Psi = [\psi]$ as an $\mathbb{R}$-vector space, $\textup{trans}_{G/Q}^{\bullet} [\upC^\bullet(\Phi^X)(\psi')]$ must be a real multiple of $\Psi$. This finishes the proof.
\end{proof}

\begin{oss}\label{oss:evaluate:everywhere:formula}
A priori Proposition~\ref{prop:baby:formula}.2 only holds almost everywhere. However, as proved by Monod~\cite[Section~1.C]{monod:lifting}, working with $\textup{L}^\infty$-cocycles on Furstenberg-Poisson boundaries one can always show that the previous formula holds everywhere~\cite[Theorem~B]{monod:lifting} 
(compare with~\cite[Section~4]{bucher2:articolo}). 
We will use this fact in the proof of Theorem~\ref{teor:coniugato:standard:embedding} in order to evaluate the formula at a given point.
\end{oss}

\subsection{Multiplicative constants and maximal measurable cocycles} \label{subsec:multiplicative:constant}

In this section we are going to introduce the notion of \emph{multiplicative constant}. This definition will allow us to introduce \emph{maximal (measurable) cocycles} and to investigate their rigidity properties.

In the situation of Setup~\ref{setup:mult:const}, let $\psi' \in \calB^\infty(Y^{\bullet+1};\bbR)^{G'}$ and let $\Psi=[\psi] \in \upH^\bullet_{cb}(G;\bbR)$ be represented by a bounded Borel cocycle $\psi \colon (G/Q)^{\bullet+1} \rightarrow \bbR$. If $\upH^\bullet_{cb}(G;\bbR)=\bbR \Psi$, then Proposition~\ref{prop:baby:formula} implies
\begin{align}\label{equation:easy:formula}
\int_{L \backslash G} \int_X \psi'(\phi(\overline{g}.\eta_1,x),&\ldots,\phi(\overline{g}.\eta_{\bullet+1},x))d\mu_X(x)d\mu(\overline{g})=\\
&=\lambda_{\psi',\psi}(\sigma)\psi(\eta_1,\ldots,\eta_{\bullet+1}) + \textup{cobound.} \nonumber \ .
\end{align}

\begin{deft}\label{def:multiplicative:constant}
The real number $\lambda_{\psi',\psi}(\sigma) \in \bbR$ appearing in Equation (\ref{equation:easy:formula}) is the \emph{multiplicative constant associated to} $\sigma, \psi', \psi$. 
\end{deft}

A particularly nice situation for the study of rigidity phenomena is when in Equation~(\ref{equation:easy:formula}) there are no coboundary terms. For this reason we are going to introduce the following notation.
\begin{deft}
We say that \emph{condition} $(\NCT)$ (no coboundary terms) is satisfied when Equation~(\ref{equation:easy:formula}) reduces to
\begin{align*}
\int_{L \backslash G} \int_X \psi'(\phi(\overline{g}.\eta_1,x),&\ldots,\phi(\overline{g}.\eta_{\bullet+1},x))d\mu_X(x)d\mu(\overline{g})=\\
=&\lambda_{\psi',\psi}(\sigma)\psi(\eta_1,\ldots,\eta_{\bullet+1}) \nonumber \ .
\end{align*}
\end{deft}
\begin{es}\label{es:when:NCT}
Standard examples in which condition $(\NCT)$ is satisfied are the followings:
\begin{enumerate}
\item Given a torsion-free lattice $L \leq G$ in a semisimple Lie group and a minimal parabolic subgroup $P \leq G$, $L$ acts doubly ergodically on the Furstenberg-Poisson boundary $G/P$~\cite[Theorem 5.6]{albuquerque99}. Hence condition $(\NCT)$ is satisfied in degree $n=2$ for bounded alternating cochains.

\item In degree $n \geq 3$, if $G=\textup{PO}(n,1)$ and $G/Q=\bbS^{n-1}$, condition $(\NCT )$ holds when we consider the real bounded cohomology twisted by the sign action \cite[Lemma 2.2]{bucher2:articolo}.
\end{enumerate}
\end{es}

\begin{oss}\label{oss:NCT}
The condition $(\NCT)$ has the following equivalent reformulation via cochains
$$
\widehat{\textup{trans}}^\bullet_{G/Q} \circ \upC^\bullet(\Phi^X)(\psi')=\lambda_{\psi',\psi}(\sigma)\psi \ . 
$$
\end{oss}

If condition $(\NCT)$ is satisfied, then there exists an explicit upper bound for the multiplicative constant $\lambda_{\psi',\psi}(\sigma)$. 

\begin{prop}\label{prop:multiplicative:upperbound}
In the situation of Setup~\ref{setup:mult:const}, let $\psi' \in \calB^\infty(Y^{\bullet+1};\bbR)^{G'}$ and let $\Psi=[\psi] \in \upH^\bullet_{cb}(G;\bbR)$ be represented by a bounded Borel cocycle $\psi \colon (G/Q)^{\bullet+1} \rightarrow \bbR$. If condition $(\textup{NCT})$ is satisfied, then we have 
$$
|\lambda_{\psi',\psi}(\sigma)| \leq \frac{\lVert \psi' \rVert_\infty}{\lVert \psi \rVert_\infty} \ .
$$
\end{prop}
\begin{proof}
By Remark~\ref{oss:NCT} we know that 
$$
\widehat{\textup{trans}}^\bullet_{G/Q} \circ \upC^\bullet(\Phi^X)(\psi')=\lambda_{\psi',\psi}(\sigma)\psi \ . 
$$
Since by Proposition \ref{prop:pullback:cohomology} $\widehat{\textup{trans}}^\bullet_{G/Q}$ and $\upC^\bullet(\Phi^X)$ are norm non-increasing maps, the left-hand side admits the following estimate
$$
\lVert \widehat{\textup{trans}}^\bullet_{G/Q} \circ \upC^\bullet(\Phi^X)(\psi') \rVert_\infty \leq \lVert \psi' \rVert_\infty \  .
$$
Hence, we get
$$
|\lambda_{\psi',\psi}(\sigma)| \lVert \psi \rVert_\infty \leq \lVert \psi' \rVert_\infty \ , 
$$
as desired. 
\end{proof}

Using the previous upper bound, we introduce the following

\begin{deft}\label{def:maximal:cocycle}
In the situation of Setup~\ref{setup:mult:const} assume that condition $(\NCT)$ is satisfied. We say that a measurable cocycle $\sigma \colon L \times X \rightarrow G'$ is \emph{maximal} if its multiplicative constant $\lambda_{\psi',\psi}(\sigma)$ attains the maximum value:
$$
\lambda_{\psi',\psi}(\sigma) = \frac{\lVert \psi' \rVert_\infty}{\lVert \psi \rVert_\infty} \ . 
$$
\end{deft}

For every representation $\pi \colon G \to G'$, we denote the restriction of $\pi$ to $L$ as $\pi |_L \colon L \to G'$. We prove now that under suitable assumptions maximal cocycles can be trivialized, i.e. they are cohomologous to a suitable representation $\pi |_L \colon L \rightarrow G'$. 

\begin{setup}\label{setup:complete:mult:const}
In the situation of Setup~\ref{setup:mult:const} assume that condition $(\textup{NCT})$ is satisfied. We also assume that
\begin{itemize} 
	\item Both $\psi'$ and $\psi$ are defined everywhere and they attain their essential supremum: 
	There exist $\eta_1,\ldots,\eta_{\bullet+1} \in G/Q$ and $y_1,\ldots,y_{\bullet+1} \in Y$ such that
$$
\psi'(y_1,\ldots,y_{\bullet+1})=\lVert \psi' \rVert_\infty \hspace{5pt} \mbox{ and } \hspace{5pt} \psi(\eta_1,\ldots,\eta_{\bullet+1})=\lVert \psi \rVert_\infty \ . 
$$
\item A \emph{maximal} map $\varphi \colon G/Q \rightarrow Y$ is a measurable map such that
$$
\psi'(\varphi(g\eta_1),\ldots,\varphi(g\eta_{\bullet+1}))=\lVert \psi' \rVert_\infty \ , 
$$
for almost every $g \in G$ and for every $\eta_1,\ldots,\eta_{\bullet+1} \in G/Q$ such that $$\psi(\eta_1,\ldots,\eta_{\bullet+1})=\lVert \psi \rVert_\infty \ .$$

\item There exists a continuous representation $\pi \colon G \rightarrow G'$ and unique continuous $\pi$-equivariant map $\Pi \colon G/Q \rightarrow Y$ which satisfies the following: Given any \emph{maximal} measurable map $\varphi \colon G/Q \rightarrow Y$, there exists a unique element $g_\varphi ' \in G'$ such that
$$
\varphi(\eta)=g_\varphi '\Pi(\eta) \ , 
$$
for almost every $\eta \in G/Q$. 

\item The $G'$-pointwise stabilizer of the map $\Pi$ is trivial, i.e. the only element $g' \in G'$ such that $g'  \Pi (x) = \Pi(x)$ for all $x \in \, G \slash Q$ is the neutral element of $G'$. We denote the previous stabilizer by $\textup{Stab}_{G'}(\Pi)$. 
\end{itemize}
\end{setup}

\begin{teor}\label{teor:coniugato:standard:embedding}
In the situation of Setup~\ref{setup:complete:mult:const} let $\pi |_L \colon L \to G'$ be the restriction of the representation $\pi \colon G \to G'$ to $L$. If the measurable cocycle $\sigma \colon L \times X \rightarrow G'$ is maximal, then $\sigma$ is cohomologous to $\pi |_L$. 
\end{teor}

\begin{oss}
More precisely, the theorem shows the existence of a measurable map $f \colon X \rightarrow G'$ such that: For all $\gamma \in L$ and almost every $x \in X$, we have
$$
\pi |_L (\gamma)=f(\gamma.x)^{-1}\sigma(\gamma,x)f(x) \ .
$$ 
\end{oss}

\begin{proof}
Since the cocycle $\sigma$ is maximal, we know that
$$
\lambda_{\psi',\psi}(\sigma) = \frac{\lVert \psi' \rVert_\infty}{\lVert \psi \rVert_\infty} \ . 
$$
Under condition $(\NCT)$, if we substitute the value of $\lambda_{\psi',\psi}(\sigma)$ in Equation (\ref{equation:easy:formula}) we get
\begin{small}
\begin{equation}\label{equation:maximal:nct:substitution}
\int_{L \backslash G} \int_X \psi'(\phi(\overline{g}.\eta_1,x),\ldots,\phi(\overline{g}.\eta_{\bullet+1},x))d\mu_X(x)d\mu(\overline{g})=\frac{\lVert \psi' \rVert_\infty}{\lVert \psi \rVert_\infty}\psi(\eta_1,\ldots,\eta_{\bullet+1})\ .
\end{equation}
\end{small}
Moreover, by assumption $\psi$ attains its essential supremum, whence there exist $\hat{\eta}_1,\ldots,\hat{\eta}_{\bullet+1} \in G/Q$ such that 
\begin{equation}\label{equation:maximum:attain}
\psi(\hat{\eta}_1,\ldots,\hat{\eta}_{\bullet+1})=\lVert \psi \rVert_\infty \ .
\end{equation} 
By Remark~\ref{oss:evaluate:everywhere:formula} we can evaluate Equation \eqref{equation:maximal:nct:substitution} at $\hat{\eta}_1,\ldots,\hat{\eta}_{\bullet+1} \in G/Q$. Hence, by Equation~(\ref{equation:maximum:attain}), we have 
\begin{equation}\label{equation:maximal:integral}
\int_{L \backslash G} \int_X \psi'(\phi(\overline{g}.\hat{\eta}_1,x),\ldots,\phi(\overline{g}.\hat{\eta}_{\bullet+1},x))d\mu_X(x)d\mu(\overline{g})=\lVert \psi' \rVert_\infty \ .
\end{equation}
This shows that
$$
\psi'(\phi(\overline{g}.\hat{\eta}_1,x),\ldots,\phi(\overline{g}.\hat{\eta}_{\bullet+1},x))=\lVert \psi' \rVert_\infty \  ,
$$
for almost every $\overline{g} \in L \backslash G$ and almost every $x \in X$. Additionally, the $\sigma$-equivariance of $\phi$ implies that in fact 
\begin{equation}\label{equation:maximal:map}
\psi'(\phi(g.\hat{\eta}_1,x),\ldots,\phi(g.\hat{\eta}_{\bullet+1},x))=\lVert \psi' \rVert_\infty 
\end{equation}
holds for almost every $g \in G$ and almost every $x \in X$. 

We can define for almost every $x \in \, X$ a map $$\phi_x \colon G/Q \rightarrow Y, \hspace{5pt} \phi_x(\eta)\coloneqq\phi(\eta,x) \ ,$$
which is measurable~\cite[Lemma 2.6]{fisher:morris:whyte} and maximal by Equation \eqref{equation:maximal:map}. Hence, by the assumptions of Setup~\ref{setup:complete:mult:const}, for almost every $x \in X$ there must exist an element $g_x \in G'$ such that 
$$
\phi_x(\eta)=g_x\Pi(\eta) \ ,
$$
for almost every $\eta \in G/Q$. This shows that $\phi_x$ lies in the $G'$-orbit of $\Pi$. In this way we get a map 
$$
\widehat{\phi} \colon X \rightarrow G'.\Pi, \ \ \ \widehat{\phi}(x)=\phi_x \ ,
$$
which is measurable~\cite[Lemma 2.6]{fisher:morris:whyte}. By Setup~\ref{setup:complete:mult:const} the stabilizer of $\Pi$ is trivial and hence the orbit $G'.\Pi$ is naturally homeomorphic to $G'$ through a map $\jmath \colon G'.\Pi \rightarrow G'$. Composing the identification $\jmath$ with the map $\widehat{\phi}$ we get a map 
$$
f \colon X \rightarrow G', \ \ \ f(x)\coloneqq (\jmath \circ \widehat{\phi})(x) \ , 
$$ 
which is defined almost everywhere and it is measurable being the composition of measurable maps (notice that the composition above gives back the element $g_x$).

We can now conclude the proof (compare with~\cite[Proposition 3.2]{sauer:articolo}). Given $\gamma \in L$, on the one hand we have
$$
\phi(\gamma.\eta,\gamma.x)=\sigma(\gamma,x)\phi(\eta,x)=\sigma(\gamma,x)f(x)\Pi(\eta) \ ,
$$
and on the other
$$
\phi(\gamma.\eta,\gamma.x)=f(\gamma.x)\Pi(\gamma.x)=f(\gamma.x)\pi |_L (\gamma)\Pi(\eta) \ .
$$
In the second equality we used the $\pi$-equivariance of the map $\Pi$. The fact that $\textup{Stab}_{G'}(\Pi)$ is trivial implies that
$$
\pi |_L (\gamma)=f(\gamma.x)^{-1}\sigma(\gamma,x)f(x) \ ,
$$
which finishes the proof.
\end{proof}

\subsection{Applications of the multiplicative formula}\label{subsec:applications:baby}

For convenience of the reader we collect here some examples of applications of Proposition~\ref{prop:baby:formula}. 

\begin{es}
Let $n \geq 3$. Let $L \leq G = \po^\circ(n, 1)$ be a torsion-free non-uniform lattice and $(X, \mu_X)$ be a standard Borel probability $L$-space. Following the notation of Setup~\ref{setup:mult:const}, we set $G' = \po^\circ(n, 1)$ and 
$Y = G / Q = \partial \mathbb{H}^n_\mathbb{R} \cong \mathbb{S}^{n-1}$, where $Q$ is a (minimal) parabolic subgroup of $G$.
Using bounded cohomology theory~\cite{Grom82, FM:Grom}, one can define the \emph{volume} $\vol(\sigma)$ of a measurable cocycle $\sigma \colon L \times X \to \po^\circ(n,1)$~\cite[Section~4.1]{moraschini:savini}. As proved by the authors~\cite[Proposition~2]{moraschini:savini}, in this setting the multiplicative constant is given by 
$$
\lambda_{\psi', \psi}(\sigma) = \frac{\vol(\sigma)}{\vol(L \backslash \mathbb{H}^n)} \ .
$$
Since condition $(\NCT)$ is satisfied for twisted real coefficients~\cite[Lemma~2.2]{bucher2:articolo}, Proposition~\ref{prop:multiplicative:upperbound} shows that the following Milnor-Wood inequality holds~\cite[Proposition~15]{moraschini:savini}
$$
| \vol(\sigma) | \leq \vol(L \backslash \mathbb{H}^n) \ . 
$$
Finally one can apply Theorem~\ref{teor:coniugato:standard:embedding} to show that if $\sigma$ is maximal, then $\sigma$ is cohomologous to the cocycle associated to the standard lattice embedding $L \rightarrow G$. In fact, one can strengthen this result: A cocycle is maximal \emph{if and only if} it is cohomologous to the cocycle associated to the standard lattice embedding~\cite[Theorem~1]{moraschini:savini}.

Similarly, one can apply an analogous strategy for studying the case of closed surfaces. The main difference is that we have to to fix a hyperbolization. Then,
maximal cocycles will be cohomologous to the given hyperbolization~\cite[Theorem~5]{moraschini:savini}
\end{es}

\begin{es}
Fix a torsion-free lattice  $L \leq G = \textup{PSL}(2,\mathbb{C})$ together with a standard Borel probability $L$-space $(X, \mu_X)$. Following the notation of Setup~\ref{setup:mult:const}, we set $G' = \textup{PSL}(n,\mathbb{C})$, $Y=\mathscr{F}(n,\mathbb{C})$ is the space of full flags, and $G / Q = \mathbb{P}^1(\mathbb{C})$. 
Here $Q$ is a (minimal) parabolic subgroup of $G$.
The second author defined the \emph{Borel invariant} $\beta_n(\sigma)$ of a measurable cocycle $\sigma \colon L \times X \to \textup{PSL}(n,\mathbb{C})$~\cite[Section~4]{savini3:articolo}. Then,  the multiplicative constant is given by~\cite[Proposition 1.2]{savini3:articolo}
$$
\lambda_{\psi', \psi}(\sigma) = \frac{\beta_n(\sigma)}{\vol(L \backslash \mathbb{H}^3)} \ .
$$
Since condition $(\NCT)$ is satisfied, Proposition~\ref{prop:multiplicative:upperbound} leads to the following Milnor-Wood inequality~\cite[Proposition~4.5]{savini3:articolo}
$$
| \beta_n(\sigma) | \leq {n+1 \choose 3}\vol(L \backslash \mathbb{H}^3) \ . 
$$
Finally, one can apply Theorem~\ref{teor:coniugato:standard:embedding} to show that if $\sigma$ is maximal, then $\sigma$ is cohomologous to the cocycle associated to the standard lattice embedding $L \rightarrow G$ composed with the irreducible representation $\pi_n:\textup{PSL}(2,\mathbb{C}) \rightarrow \textup{PSL}(n,\mathbb{C})$. In fact also the converse holds true~\cite[Theorem~1.1]{savini3:articolo}.
\end{es}

\section{Cartan invariant of measurable cocycles of complex hyperbolic lattices}\label{sec:cartan:invariant}

Let $\Gamma \leq \pu(n,1)$ be a torsion-free lattice with $n \geq 2$ and let $(X,\mu_X)$ be a standard Borel probability $\Gamma$-space. In this section we are going to define the \emph{Cartan invariant} $i(\sigma)$ associated to a measurable cocycle $\sigma \colon \Gamma \times X \rightarrow \pu(m,1)$. Then, when $\sigma$ is non elementary, we will express the Cartan invariant as a multiplicative constant (Proposition \ref{prop:cartan:multiplicative:cochains}). This interpretation allows us to deduce many properties of the Cartan invariant for non-elementary measurable cocycles. 

We recall here just few notions of complex hyperbolic geometry that we will need in the sequel. We refer the reader to Goldman's book~\cite{Goldmancomplex}
for a complete discussion about this topic. Let $\mathbb{H}^n_{\mathbb{C}}$ be the complex hyperbolic space. For every $k \in \{ 0, \ldots, n\}$ a \emph{$k$-plane} is a totally geodesic copy of $\bbH^{k}_{\bbC}$ holomorphically embedded in $\bbH^n_{\bbC}$. 
When $k=1$, a $1$-plane is simply a \emph{complex geodesic}. Similarly, a \emph{$k$-chain} is the boundary of a $k$-plane in $\partial_\infty \bbH^n_{\bbC}$, i.e. it is an embedded copy of $\partial_\infty \bbH^k_{\bbC}$. When $k = 1$, we will just call them \emph{chains}. Since a chain is completely determined by any two of its points, two distinct chains are either disjoint or they meet exactly in one point. 

Let us consider the Hermitian triple product 
$$
\langle \cdot, \cdot, \cdot \rangle \colon (\bbC^{n,1})^3 \rightarrow \bbC, \hspace{5pt} \langle z_1, z_2, z_3 \rangle \coloneqq h(z_1,z_2)h(z_2,z_3)h(z_3,z_1) \ .
$$
If we denote by $(\partial_\infty \bbH^n_{\bbC})^{(3)}$ the set of triples of distinct points in the boundary at infinity, we can defined the following function
$$
c_n \colon (\partial_\infty \bbH^n_{\bbC})^{(3)} \rightarrow [-1,1], \hspace{5pt} c_n(\xi_1,\xi_2,\xi_3) \coloneqq \frac{2}{\pi} \arg \langle z_1,z_2,z_3 \rangle \ .
$$
Here $\xi_i=[z_i]$ and we choose the branch of the argument function such that $\arg(z) \in [-\pi/2,\pi/2]$. 
Then, we can extend $c_n$ to a $\pu(n,1)$-invariant alternating Borel cocycle on the whole $(\partial_\infty \bbH^n_{\bbC})^3$. 
Moreover, $|c_n(\xi_1,\xi_2,\xi_3)|=1$ if and only if $\xi_1,\xi_2,\xi_3 \in \partial_\infty \bbH^n_{\bbC}$ are distinct and they lie on the same chain \cite[Section 3]{BIW09}. 

\begin{deft}
The cocycle 
$$
c_n \in \calB^\infty_{\textup{alt}}((\partial_\infty \bbH^n_{\bbC})^3;\bbR)^{\pu(n,1)} 
$$
is called \emph{Cartan cocycle}.
\end{deft}
\begin{oss}\label{oss:cartan:cocycle:det:class}
The Cartan cocycle $c_n$ canonically determines a class in $\upH^2_{cb}(\pu(n,1);\bbR)$ via the map defined in Equation~(\ref{eq:canonical:map:B:L}).
\end{oss}

Let $\omega_n \in \Omega^2(\bbH^n_{\bbC})$ be the K\"ahler form, which is a $\pu(n,1)$-invariant $2$-form. 
By the Van Est isomorphism~\cite[Corollary~7.2]{guichardet} the space $\Omega^2(\bbH^n_{\bbC})^{\pu(n,1)}$ is isomorphic to $\upH^2_c(\pu(n,1);\bbR)$. We call \emph{K\"ahler class} the element $\kappa_n \in \, \upH^2_c(\pu(n,1);\bbR)$ corresponding to $\omega_n$ via the previous isomorphism. Since the K\"ahler class is bounded, $\kappa_n$ lies in the image of the comparison map 
$$
\textup{comp}^2 \colon \upH^2_{cb}(\pu(n,1);\bbR) \rightarrow \upH^2_c(\pu(n,1);\bbR) \ .
$$
Hence, there exists a class $\kappa_n^b \in \upH^2_{cb}(\pu(n,1);\bbR)$ which is sent to $\kappa$ under $\textup{comp}^2$. Since the group $\upH^2_{cb}(\pu(n,1);\bbR)$ is one dimensional, we can assume that $\kappa_n^b$ is its generator as real vector space. 
The relation between the Cartan cocycle and the bounded K\"ahler class is the following (Remark~\ref{oss:cartan:cocycle:det:class})
$$
[c_n]=\frac{\kappa^b_n}{\pi} \in \upH^2_{cb}(\pu(n,1);\bbR) \ . 
$$
\begin{oss}\label{oss:cartan:repr:kaehler}
The previous equality shows that the cocycle $\pi c_n$ is a representative of the bounded K\"ahler class. 
\end{oss}

\begin{setup}\label{setup:cartan:invariant}
Let $n \geq 2$. We assume the following
\begin{itemize}
\item Let $\Gamma \leq \pu(n,1)$ be a torsion-free lattice;

\item Let $(X,\mu_X)$ be a standard Borel probability $\Gamma$-space;

\item Let $\sigma \colon \Gamma \times X \rightarrow \pu(m,1)$ be a measurable cocycle.
\end{itemize}
\end{setup}
In the previous situation $\sigma$ induces a map in bounded cohomology (Lemma~\ref{lem:pullback:map:cocycle:cohomology})
$$
\upH^2_b(\sigma): \textup{H}^2_{cb}(\pu(m,1);\bbR) \rightarrow \upH^2_b(\Gamma;\bbR) \ .
$$
Moreover, since $\Gamma$ is a lattice, there exists a transfer map (Definition~\ref{def:trans:map})
$$
\trans_\Gamma^2:\upH^2_b(\Gamma;\bbR) \rightarrow \upH^2_{cb}(\pu(n,1);\bbR) \ . 
$$
Composing the two maps above we can give the following
\begin{deft}\label{def:cartan:invariant}
In the situation of Setup~\ref{setup:cartan:invariant}, the \emph{Cartan invariant associated to the cocycle $\sigma$} is the real number $i(\sigma)$ appearing in the following equation
\begin{equation}\label{eq:cartan:no:boundary:map}
\trans_\Gamma^2 \circ \upH^2_b(\sigma)(\kappa^b_m)=i(\sigma)\kappa^b_n \ . 
\end{equation}
\end{deft}
\begin{oss}\label{oss:cartan:inv:cocycle:1:well:def}
The previous formula is well-defined since $\upH^2_{cb}(\pu(n,1);\bbR) \cong \mathbb{R}\kappa_n^b$.
\end{oss}

We explain now how to compute the Cartan invariant in terms of a boundary map associated to $\sigma$. This will show that the Cartan invariant is a multiplicative constant in the sense of Definition \ref{def:multiplicative:constant}. 

First recall that every non-elementary measurable cocycle $\sigma \colon \Gamma \times X \rightarrow \pu(m,1)$ admits an \emph{essentially unique} boundary map~\cite[Proposition 3.3]{MonShal0} (Remark~\ref{oss:esistenza:mappe:bordo}). Here essentially unique means that two boundary maps coincide on a full measure set. As noticed by Monod and Shalom, the non-elementary condition means that the group of the real points of the algebraic hull of $\sigma$ (Definition \ref{def:alg:hull}) is a non-elementary subgroup of $\pu(n,1)$.

By Lemma \ref{lem:pullback:implemented:boundary:map} the existence of a boundary map implies that the pullback map $\upH^2_b(\sigma)$ coincides with the following composition
$$
\upH^2_b(\sigma)=\upH^2(\Phi^X) \circ \mathfrak{c}^2 \ ,
$$
where $\mathfrak{c}^2$ and $\upH^2(\Phi^X)$ are the maps introduced in Equation~(\ref{eq:canonical:map:B:L}) and Definition~\ref{def:pullback:boundary}, respectively. Thus, Equation~\eqref{eq:cartan:no:boundary:map} is equivalent to 
\begin{equation}\label{eq:cartan:multiplicative:constant} 
\trans_\Gamma^2 \circ \upH^2(\Phi^X)([\pi c_m])=i(\sigma) \kappa^b_n \ .
\end{equation}
This shows that the Cartan invariant is a \emph{multiplicative constant} in the sense of Definition \ref{def:multiplicative:constant}. We are going to prove that Equation \ref{eq:cartan:multiplicative:constant} actually holds at the levels of cochains. 

\begin{repprop}{prop:cartan:multiplicative:cochains}
In the situation of Setup~\ref{setup:cartan:invariant}, let $\sigma$ be a non-elementary measurable cocycle with boundary map $\phi \colon \partial_\infty \bbH^n_{\bbC} \times X \rightarrow \partial_\infty \bbH^m_{\bbC}$. Then, for every triple of pairwise distinct points $\xi_1,\xi_2,\xi_3 \in \partial_\infty \bbH^n_{\bbC}$, we have
\begin{small}
\begin{equation}\label{eq:cartan:multiplicative:cochains}
i(\sigma)c_n(\xi_1,\xi_2,\xi_3)=\int_{\Gamma \backslash \pu(n,1)} \int_X c_m(\phi(\overline{g}\xi_1,x),\phi(\overline{g}\xi_2,x),\phi(\overline{g}\xi_3,x))d\mu(\overline{g})d\mu_X(x) \ .
\end{equation}
\end{small}
Here $\mu$ is a $\pu(n,1)$-invariant probability measure on the quotient $\Gamma \backslash \pu(n,1)$. 
\end{repprop}

\begin{proof}
We already know that $\pi c_n$ is a representative of $\kappa^b_n$ (Remark~\ref{oss:cartan:repr:kaehler}). Moreover, since $\Gamma$ acts doubly ergodically on $\partial_\infty \bbH^n_{\bbC}$, there are no essentially bounded $\Gamma$-invariant alternating functions on $(\partial_\infty \bbH^n_{\bbC})^2$. 
Hence, if we rewrite Equation \eqref{eq:cartan:multiplicative:constant} in terms of cochains, we obtain a formula
$$
\widehat{\trans}^2_{\partial_\infty \bbH^n_{\bbC}} \circ \upC^2(\Phi^X)(\pi c_m)=i(\sigma)(\pi c_n) \ ,
$$
without coboundaries. Here $\widehat{\trans}^2_{\partial_\infty \bbH^n_{\bbC}}$ is the map introduced at the end of Section \ref{subsec:transfer:maps}. Since the constant $\pi$ appears on both sides, we get
$$
i(\sigma)c_n(\xi_1,\xi_2,\xi_3)=\int_{\Gamma \backslash \pu(n,1)} \int_X c_m(\phi(\overline{g}\xi_1,x),\phi(\overline{g}\xi_2,x),\phi(\overline{g}\xi_3,x))d\mu(\overline{g})d\mu_X(x) \ ,
$$
for almost every $\xi_1,\xi_2,\xi_3 \in \partial_\infty \bbH^n_{\bbC}$. The fact that the equation is in fact true for \emph{every triple of pairwise distinct points} can be proved by following \emph{verbatim} Pozzetti's proof~\cite[Lemma 2.11]{Pozzetti}. This finishes the proof.
\end{proof}

The interpretation of the Cartan invariant as a multiplicative constant has many consequences. 
For instance, the Cartan invariant of measurable cocyles extends the one for representations introduced by Burger and Iozzi~\cite{BIcartan}. 

\begin{prop}\label{prop:cartan:rep}
In the situation of Setup~\ref{setup:cartan:invariant}, let $\rho \colon \Gamma \rightarrow \pu(m,1)$ be a non-elementary representation and let $\sigma_\rho \colon \Gamma \times X \rightarrow \pu(m,1)$ be the associated measurable cocycle. Then, we have
$$
i(\sigma_\rho)=i(\rho) \ .
$$
\end{prop}

\begin{proof}
Since $\rho$ is non elementary, $\rho$ admits an essentially unique boundary map $\varphi:\partial_\infty \bbH^n_{\bbC} \rightarrow \partial_\infty \bbH^m_{\bbC}$~\cite[Corollary 3.2]{burger:mozes}. Hence, one can define a boundary map for the cocycle $\sigma_\rho$ as
$$
\phi \colon \partial_\infty \bbH^n_{\bbC} \times X \rightarrow \partial_\infty \bbH^m_{\bbC} \ , \ \phi(\xi,x):=\phi(\xi) \ ,
$$
for almost every $\xi \in \partial_\infty \bbH^n_{\bbC}$ and almost every $x \in X$. This readily implies that
\begin{align*}
i(\sigma_\rho)c_n&=\widehat{\trans}^2_{\partial_\infty \bbH^n_{\bbC}} \circ \upC^2(\Phi^X)(c_m) \\
&=\widehat{\trans}^2_{\partial_\infty \bbH^n_{\bbC}} \circ \upC^2(\varphi)(c_m) \\
&=i(\rho)c_n \ .
\end{align*}
The first equality comes from Proposition~\ref{prop:cartan:multiplicative:cochains}, the second one is due to Proposition~\ref{prop:pullback:coc:vs:repr} and finally the last equality is proved by Burger and Iozzi~\cite[Lemma 5.3]{BIgeo}. 
\end{proof}

We conclude this section by showing that the Cartan invariant is constant along cohomology classes and it has bounded  absolute value.
\begin{prop}\label{prop:cartan:cohomology:bound}
In the situation of Setup~\ref{setup:cartan:invariant}, let $\sigma$ be a non-elementary measurable cocycle with boundary map $\phi:\partial_\infty \bbH^n_{\bbC} \times X \rightarrow \partial_\infty \bbH^m_{\bbC}$. Then, the following hold
\begin{enumerate} 
	\item The Cartan invariant $i(\sigma)$ is constant along the $\pu(m,1)$-cohomology class of $\sigma$.
	\item $|i(\sigma)| \leq 1$.
\end{enumerate}
\end{prop}

\begin{proof}
\emph{Ad.~1} The first statement follows by Proposition~\ref{prop:invariance:cohomology}. Indeed if $f \colon X \rightarrow \pu(m,1)$ is a measurable map, then Proposition~\ref{prop:invariance:cohomology} shows that 
$$
\upH^2(f.\Phi^X)([\pi c_m])=\upH^2(\Phi^X)([\pi c_m]) \ .
$$
Hence, we get
$$
i(f.\sigma)\kappa^b_n=\trans_\Gamma^2 \circ \upH^2(f.\Phi^X)([\pi c_m])=\trans_\Gamma^2 \circ \upH^2(\Phi^X)([\pi c_m])=i(\sigma) \kappa^b_n \ .$$
This shows that
$
i(f.\sigma)=i(\sigma)
$,
as desired. 

\emph{Ad.~2} Since the Cartan invariant is a multiplicative constant and condition $(\NCT)$ is satisfied, Proposition \ref{prop:multiplicative:upperbound} implies
$$
|i(\sigma)| \leq 1 \ .
$$
Here we used the fact that $\lVert c_n \rVert_\infty=\lVert c_m \rVert_\infty=1$. 
\end{proof}

The second item of the previous proposition leads to the following definition (compare with Definition \ref{def:maximal:cocycle}).

\begin{deft}\label{def:maximal:cartan}
In the situation of Setup~\ref{setup:cartan:invariant}, a non-elementary cocycle $\sigma$ is \emph{maximal} if $i(\sigma)=1$. 
\end{deft}

\section{Totally real cocycles}\label{sec:tot:real}

In this section we introduce the notion of \emph{totally real cocycles}. 
Our definition extends the one by Burger and Iozzi~\cite{BIreal} for representations. We aim to investigate the relation between the vanishing of the Cartan invariant and the condition of being totally real. We will show that totally real cocycles have trivial Cartan invariant. On the other hand, it is natural to ask whether the converse is also true. We partially answer to this question by showing that ergodic cocycles inducing the trivial map in bounded cohomology are totally real. 

\begin{deft}\label{def:totally:real}
In the situation of Setup~\ref{setup:cartan:invariant} we denote by $\mathbf{L}$ the algebraic hull of $\sigma$.  Let $L \coloneqq \mathbf{L}(\mathbb{R})^\circ$ be the connected component of the identity of the group of real points of $\mathbf{L}$. 
A measurable cocycle $\sigma$ is \emph{totally real} if for some $1 \leq k \leq m$ we have
$$
L \subseteq \calN_{\textup{PU}(m,1)}(\mathbb{H}^k_{\mathbb{R}}) \ ,
$$
where $\mathbb{H}^k_{\mathbb{R}} \subset \mathbb{H}^m_{\mathbb{C}}$ is a totally geodesic copy of the real hyperbolic $k$-space. Here $\calN_{\pu(m,1)}(\mathbb{H}^k_{\mathbb{R}})$ denotes the subgroup of $\pu(m,1)$ preserving the fixed copy of $\bbH^k_{\bbR}$, i.e. $g (\bbH^k_{\bbR}) \subset \bbH^k_{\bbR}$ for every $g \in \calN_{\textup{PU}(m,1)}(\mathbb{H}^k_{\mathbb{R}})$.
\end{deft}

\begin{oss}\label{oss:tot:real:cohomologous}
By the definition of algebraic hull (Definition~\ref{def:alg:hull}) every totally real cocycle $\sigma$ is cohomologous to a cocycle $\widehat{\sigma}$ whose image is contained in $L$. 
Additionally, $\calN_{\textup{PU}(m,1)}(\mathbb{H}^k_{\mathbb{R}})$ is an almost direct product of a compact subgroup $K \leq \pu(m,1)$ with an embedded copy of $\textup{PO}(k,1)$ inside $\textup{PU}(m,1)$. Hence, the cocycle $\widehat{\sigma}$ preserves the totally geodesic copy $\mathbb{H}^k_{\mathbb{R}} \subset \mathbb{H}^m_{\mathbb{C}}$ stabilized by $L$. 
\end{oss} 

In the sequel we need the following

\begin{lem}\label{lem:no:two:coboundary}
Let $\Gamma \leq \textup{PU}(n,1)$ be a torsion-free lattice, with $n \geq 2$, and let $(X,\mu_X)$ be a standard Borel probability space. Then 
$$
\upH^2_b(\Gamma;\upL^\infty(X)) \cong \mathcal{Z}\upL^\infty_{\textup{w}^\ast,\textup{alt}}((\partial_\infty \bbH^n_{\bbC})^3; \upL^\infty(X))^\Gamma \ .
$$
Here, the letter $\mathcal{Z}$ denotes the set of cocycles and the subscript \emph{alt} denotes the restrictions to alternating essentially bounded weak-${}^*$ measurable functions. 
\end{lem}

\begin{proof}
For every $k \in \bbN$ we have the following
$$
\upL^\infty_{\textup{w}^\ast}((\partial_\infty \bbH^n_{\bbC})^k; \upL^\infty(X))^\Gamma \cong \upL_{\textup{w}^\ast}^\infty((\partial_\infty \bbH^n_{\bbC})^k \times X;\bbR)^\Gamma \ ,
$$
where $\Gamma$ acts on $(\partial_\infty \bbH^n_{\bbC})^k \times X$ diagonally \cite[Corollary 2.3.3]{monod:libro}. Moreover, every $\Gamma$-invariant essentially bounded weak-${}^*$ measurable function on $(\partial_\infty \bbH^n_{\bbC})^2\times X$ must be essentially constant~\cite[Proposition 2.4]{MonShal0}. Since an alternating function that is constant vanishes, we have that  
$$
\upL^\infty_{\textup{w}^\ast,\textup{alt}}((\partial_\infty \bbH^n_{\bbC})^2; \upL^\infty(X))^\Gamma = 0 \ .
$$
This shows that there are no coboundaries in dimension two, whence the thesis. 
\end{proof}

The notion of totally real cocycles is strictly related to the vanishing of the Cartan invariant. This correspondence is described by the following result which is a suitable adaptation of a result by Burger and Iozzi~\cite[Theorem 1.1]{BIreal} to the case of measurable cocycles.

\begin{repteor}{teor:totally:real}
In the situation of Setup~\ref{setup:cartan:invariant}, let $\sigma$ be a non-elementary measurable cocycle with boundary map $\phi:\partial_\infty \mathbb{H}^n_{\bbC} \times X \rightarrow \partial_\infty \bbH^m_{\bbC}$. Then, the following hold

\begin{enumerate}
	\item  If $\sigma$ is totally real, then $i(\sigma)=0$;
	\item  If $X$ is $\Gamma$-ergodic and $\textup{H}^2(\phi)([c_m])=0$, then $\sigma$ is totally real.
\end{enumerate}
\end{repteor}

\begin{proof}
\emph{Ad~1.} Let $\mathbf{L}$ be the algebraic hull of $\sigma$ and let $L=\mathbf{L}(\bbR)^\circ$ be the connected component of the real points of $\mathbf{L}$ containing the identity.  By Remark~\ref{oss:tot:real:cohomologous}, there exists a cocycle $\widehat{\sigma}$ cohomologous to $\sigma$ such that $$\widehat{\sigma} (\Gamma \times X) \subset L \subset \calN_{\textup{PU}(m,1)}(\bbH^k_{\bbR}) \ ,$$
for some $1 \leq k \leq m$. 
Since $\widehat{\sigma}$ is cohomologous to $\sigma$, it admits a boundary map $\widehat{\phi}$ (Remark~\ref{oss:twisted:boundary:map}). 
Hence, since the Cartan invariant is constant along the $\pu(m,1)$-cohomology class of $\sigma$, it is sufficient to show that $i(\widehat{\sigma}) = 0$
(Proposition \ref{prop:cartan:cohomology:bound}). 

By Remark~\ref{oss:tot:real:cohomologous} the cocycle $\widehat{\sigma}$ also preserves the totally geodesic copy of $\mathbb{H}^k_{\bbR}$ stabilized by $L$, whence it preserves the boundary at infinity $\partial_\infty \mathbb{H}^k_{\bbR}$. We identify $\partial_\infty \mathbb{H}^k_{\bbR}$ with a $(k-1)$-dimensional sphere $\widehat{\calS} \subset \bbH^m_{\bbC}$ as explained in Section~\ref{sec:cartan:invariant}. Hence, the boundary map $\widehat{\phi}$ takes values in $\widehat{\calS}$, that is 

$$
\widehat{\phi} \colon \partial_\infty \bbH^n_{\bbC} \times X \rightarrow \widehat{\calS} \ .
$$

For almost every $x \in X$, we then define 
\begin{align*}
\widehat{\phi}_x \colon \partial_\infty \bbH^n_{\bbC} \rightarrow \widehat{\calS} \\
\widehat{\phi}_x(\xi) \coloneqq \widehat{\phi}(\xi,x) \ .
\end{align*}
The map $\widehat{\phi}_x$ is measurable for almost every $x \in X$~\cite[Lemma 2.6]{fisher:morris:whyte}. By Proposition \ref{prop:cartan:multiplicative:cochains} we have
$$
\int_{\Gamma \backslash \pu(n,1)} \int_X c_m(\widehat{\phi}_x(\overline{g}.\xi_1),\widehat{\phi}_x(\overline{g}.\xi_2),\widehat{\phi}_x(\overline{g}.\xi_3))d\mu_X(x)d\mu(\overline{g})=i(\widehat{\sigma}) c(\xi_1,\xi_2,\xi_3) \ ,
$$
for almost every $\xi_1,\xi_2,\xi_3 \in \partial_\infty \bbH^n_{\bbC}$. Here $\mu$ is the $\pu(n,1)$-invariant probability measure on $\Gamma \backslash \pu(n,1)$. Since $\widehat{\phi}_x$ takes values into the sphere $\widehat{\calS}$ for almost every $x \in X$, we have that
$$
c_m(\widehat{\phi}_x(\overline{g}.\xi_1),\widehat{\phi}_x(\overline{g}.\xi_2),\widehat{\phi}_x(\overline{g}.\xi_3))=0 \ ,
$$
for almost every $x \in X$ and almost every $\overline{g} \in \Gamma \backslash \pu(n,1)$~\cite[Corollary 3.1]{BIreal}. Thus, $i(\widehat{\sigma})$ vanishes. Since $i(\sigma) = i(\widehat{\sigma})$ we get the thesis. 

\emph{Ad~2.} Since $\upC^2(\phi)$ is a cochain map (Lemma~\ref{lemma:pullback:cochain}), it induces a map $\upH^2(\phi)$ in cohomology.
If $\upH^2(\phi)([c_m])=0$, then we have
$$
\upH^2(\phi)([c_m])=\left[ \upC^2(\phi)(c_m)\right] =0 \ .
$$
Since by Lemma \ref{lem:no:two:coboundary} there are no $\upL^\infty(X)$-coboundaries in degree $2$, we have that
$$
\upC^2(\phi)(c_m)=0 \ .
$$
More precisely 
\begin{equation}\label{equation:vanishing:cartan:ess:image}
c_m(\phi_x(\xi_1),\phi_x(\xi_2),\phi_x(\xi_3))=0 \ ,
\end{equation}
for almost every $x \in X$ and almost every $\xi_1,\xi_2,\xi_3 \in \partial_\infty \bbH^n_{\bbC}$. For almost every point $x \in \, X$, let us denote by $E_x$ the essential image of $\phi_x$, i.e. the support of the push-forward measure $(\phi_x)_\ast \nu$, where $\nu$ is the standard round measure on $\partial_\infty \bbH^n_{\bbC}$. 

We have just proved in Equation~(\ref{equation:vanishing:cartan:ess:image}) that for almost every $x \in \, X$ the Cartan cocycle $c_m$ vanishes over $E_x$. Hence, as proved by Burger and Iozzi~\cite[Corollary 3.1]{BIreal}, for almost every $x \in X$, there exists an integer $1 \leq k(x) \leq m$ and a real $(k(x)-1)$-sphere $\calS_x$ embedded in $\partial_\infty \bbH^m_{\bbC}$, such that 
$$
E_x \subseteq \calS_x \ .
$$
Moreover, we can choose $\calS_x$ to be minimal with respect to the inclusion. We claim now that 
\begin{equation} \label{eq:dim:inv}
\calS_{\gamma.x}=\sigma(\gamma,x)\calS_x \ ,
\end{equation}
for every $\gamma \in \Gamma$ and almost every $x \in X$. First, the definition of $E_x$ and the $\sigma$-equivariance of  $\phi$ imply that 
$$
E_{\gamma.x}=\sigma(\gamma,x) E_x\ ,
$$
for every $\gamma \in \Gamma$ and almost every $x \in X$. Hence, we have
$$
E_{\gamma.x} =\sigma(\gamma,x)E_x \subset \sigma(\gamma,x)\calS_x \ .
$$
Thus, the minimality assumption shows that 
$$
\calS_{\gamma.x} \subset \sigma(\gamma,x)\calS_x \ .
$$

By interchanging the role of $E_{\gamma,x}$ and $E_x$, we get the claim. As a consequence $k(\gamma.x)=k(x)$ for every $\gamma \in \Gamma$ and almost every $x \in X$. 
The ergodicity assumption on the space $(X,\mu_X)$ then implies that almost all the spheres have the same dimension, i.e. $k(x) = k \in \, \mathbb{N}$ for almost every $x \in \, X$.

Let us now denote by $\textup{Sph}^{k-1}(\partial_\infty \bbH^m_{\bbC})$ the space of $(k-1)$-spheres embedded in the boundary at infinity $\partial_\infty \bbH^m_{\bbC}$. 
Since the action of $\pu(m,1)$ on $(k-1)$-spheres is transitive, $\textup{Sph}^{k-1}(\partial_\infty \bbH^m_{\bbC})$ is a $\pu(m,1)$-homogeneous space.
Let $G_0=\calN_{\pu(m,1)}(\calS_0)$ be the subgroup of $\pu(m,1)$ preserving a fixed $(k-1)$-sphere $\calS_0$. Then we can define a map
$$
\calS \colon X \rightarrow \textup{Sph}^{k-1}(\partial_\infty \bbH^m_{\bbC}) \ , \ \ \calS(x)=\calS_x \ ,
$$
which is measurable because $\phi_x$ varies measurably with respect to $x \in X$ \cite[Lemma 2.6]{fisher:morris:whyte}. 
Since $\textup{Sph}^{k-1}(\partial_\infty \bbH^m_{\bbC}) \cong \textup{PU}(m,1)/G_0$, we can compose the previous map with a measurable section~ \cite[Corollary A.8]{zimmer:libro} 
$$
s \colon\textup{PU}(m,1)/G_0 \rightarrow \textup{PU}(m,1) \ .
$$
Let $f \colon X \rightarrow \pu(m,1)$  be the composition $s \circ \calS$. Since $f$ is a composition of measurable maps, it is measurable. Moreover, by construction we have 
$$
\calS_x=f(x)\calS_{0} \ ,
$$
for almost every $x \in X$. 

Let us consider now the $f$-twisted cocycle $\sigma_0=f.\sigma$ associated to $\sigma$ (Definition \ref{def:cohomology:cocycle}). On the one hand we have that 
$$
\calS_{\gamma.x}=f(\gamma.x)\calS_0 \ ,
$$
on the other
$$
\calS_{\gamma.x}=\sigma(\gamma,x)f(x)\calS_0 \ .
$$
Hence $\sigma_0$ preserves $\calS_0$. This implies that $\sigma_0(\Gamma \times X) \subset G_0$. If $\mathbf{L}$ denotes the algebraic hull of $\sigma$ (which is the same for $\sigma_0$) and $L=\mathbf{L}(\bbR)^\circ$, we get
$$
L \subseteq G_0 \ ,
$$
whence the thesis.
\end{proof}

\begin{oss}
Unfortunately, we are not able to show that \emph{Ad.~2} actually  provides a complete converse to \emph{Ad.1}. Indeed, it is not unlikely that the vanishing of the pullback $\upH^2(\phi)([c_m])$ is in fact a stronger condition than the vanishing of the Cartan invariant associated to the cocycle $\sigma$.
A priori the condition $i(\sigma)=0$ does not necessarily imply that the pullback induced by $\phi$ vanishes on $c_m$. 
However, at the moment we are not able to construct an explicit example of such situation.

On the other hand, our formulation of Theorem~\ref{teor:alg:hull}.2 suitably extends Burger-Iozzi's result~\cite[Theorem 1.1]{BIreal} in the setting of measurable cocycles.
Indeed when $\sigma$ is actually cohomologous to a non-elementary representation $\rho$, the pullback along $\phi$ boils down to the pullback along $\rho$ (Proposition~\ref{prop:pullback:coc:vs:repr}). Thus, we completely recover~\cite[Theorem 1.1]{BIreal} in this particular situation. 
\end{oss}

\section{Rigidity of the Cartan invariant}\label{sec:cartan:rigidity}

In this section we discuss some rigidy results which can be deduced using the Cartan invariant of measurable cocycles. 
We first study the algebraic hull (Definition~\ref{def:alg:hull}) of cocycles whose pullback does not vanish. 
Then, we characterize maximal measurable cocycles (Definition~\ref{def:maximal:cartan}).

We begin with the following result which is a suitable extension of Burger-Iozzi's result for representation~\cite[Theorem 1.2]{BIreal}:
\begin{repteor}{teor:alg:hull}
In the situation of Setup~\ref{setup:cartan:invariant}, let $\sigma$ be a non-elementary measurable cocycle with boundary map $\phi \colon \partial_\infty \bbH^n_{\bbC} \times X \rightarrow \partial_\infty \bbH^m_{\bbC}$. Let $\mathbf{L}$ be the algebraic hull of $\sigma$ and let $L=\mathbf{L}(\bbR)^\circ$ be the connected component of the identity of the real points. 

If $\upH^2(\Phi^X)([c_m])\neq 0$, then $L$ is an almost direct product $K \cdot M$, where $K$ is compact and $M$ is isomorphic to $\textup{PU}(p,1)$ for some $1 \leq p \leq m$. 

In particular, the symmetric space associated to $L$ is a totally geodesically embedded copy of $\bbH^p_{\bbC}$ inside $\bbH^m_{\bbC}$. 
\end{repteor}

\begin{proof}
Since $\textup{H}^2(\Phi^X)([c_m])$ does not vanish, the restriction of the bounded K\"ahler class $\kappa^b_m$ to $\upH^2_b(L; \mathbb{R})$ does not vanish (Remark~\ref{oss:cartan:repr:kaehler}). Thus, $L$ cannot be amenable. Since $\textup{PU}(m,1)$ has real rank one, $L$ is reductive with semisimple part $M$ of rank one. We denote by $K$ the compact term in the decomposition of $L$.
Since $K$ compact, whence amenable, we have that $\kappa^b_m|K=0$ (see Remark \ref{oss:notation:restriction} for the notation). Hence, $\kappa^b_m|L \neq 0$ implies $\kappa^b_m|M \neq 0$. As a consequence we have
$$
\upH^2_c(M;\bbR) \cong \upH^2_{cb}(M ; \bbR) \neq 0 \ .
$$
Thus, $M$ is a group of Hermitian type. Then, since $M$ has real rank one, we have
$$
M \cong \textup{PU}(p,1) \ ,
$$
for some $1 \leq p \leq m$. Take an isomorphism $\pi:\textup{PU}(p,1) \rightarrow M$ such that $\upH^2_{cb}(\pi)(\kappa^b_m)=\lambda \kappa^b_p$ for some $\lambda >0$. If $m \geq 2$, the map $\pi:\pu(p,1) \rightarrow \pu(m,1)$ corresponds to a totally geodesic embedding $\bbH^p_{\bbC}$ inside $\bbH^m_{\bbC}$ (which is holomorphic by the positivity of $\lambda$). When $m=1$, the group $\pi(\pu(1,1))$ cannot correspond to a totally real embedding, otherwise $\lambda=0$ by Theorem \ref{teor:totally:real}. Hence it must correspond to a complex geodesic and the statement is proved. 
\end{proof}

Among the cocycles with non-trivial pullback, maximal ones can be completely characterized. Maximal cocycles always admit a(n essentially unique) boundary map. Indeed they are non-elementary, since the latters have trivial Cartan invariant.  

\begin{repteor}{teor:cartan:rigidity}
In the situation of Setup~\ref{setup:cartan:invariant}, let $(X, \mu)$ be ergodic and let $\sigma$ be a maximal cocycle. 
Let $\mathbf{L}$ be the algebraic hull of $\sigma$ and let $L=\mathbf{L}(\bbR)^\circ$ be the connected component of the identity of the real points. 

Then, the following hold
\begin{enumerate}
\item $m \geq n$; 

\item $L$ is an almost direct product $\textup{PU}(n,1) \cdot K$, where $K$ is compact;

\item $\sigma$ is cohomologous to the cocycle $\sigma_i$ associated to the standard lattice embedding $i \colon \Gamma \to \pu(m, 1)$ (possibly modulo the compact subgroup $K$ when $m >n$). 
\end{enumerate}
\end{repteor}
\begin{proof}
One can restrict the image of $\sigma$ to its algebraic hull, which is completely characterized by Theorem \ref{teor:alg:hull}. In this way we obtain a Zariski dense cocycle. The thesis now follows \emph{verbatim} as in the proof of~\cite[Theorem 2]{sarti:savini}.
\end{proof}

\section{Concluding remarks}\label{sec:concluding:remarks}

We conclude the paper with a short list of comments that relate the notion of Cartan invariant with more recent results in this field. These results have been obtained by combining the theory developed in this paper with new insights. 

Recently one of the authors has proved a statement analogous to~Theorem \ref{teor:cartan:rigidity} but with completely different techniques~\cite[Theorem 1.2]{savini:natural:cocycle}.
 The main new ingredient was the existence of natural maps associated to a measurable cocycles. Natural maps exist for ergodic Zariski dense cocycles,
 e.g. cocycles arising from ergodic couplings \cite[Lemma 3.6]{savini:tautness}. The existence of natural maps also played an important role in the recent proof of the $1$-tautness conjecture for $\pu(n,1)$, with $n \geq 2$~\cite[Theorem 1]{savini:tautness}. This result provides a nice classification of discrete groups that are $\upL^1$-measure equivalent to a lattice $\Gamma \leq \pu(n,1)$. 
 In that situation the key point is to show that measurable cocycles arising from ergodic self-couplings associated to a uniform lattice $\Gamma \leq \pu(n,1)$ are maximal. This then implies that they are cohomologous to the standard lattice embedding by using the results of this paper. The notion of maximality introduced in \cite{savini:tautness} agrees with the one in Definition \ref{def:maximal:cartan}. This provides a wide family of cocycles which do not come from representations but they are cohomologous to them. 

Unfortunately the authors were not able to prove the $1$-tautness conjecture directly with the use of the Cartan invariant. The main obstruction to this approach concerns the study of cup products of bounded cohomology classes, which is a highly non trivial subject~\cite{Heuer:cup, BM:cup, AmonBuch}.

The study of lattices in $\pu(1,1)$ was separated from this project because it contained some additional difficulties. Recently one of the authors used some ideas of this paper to provide a complete characterization of the algebraic hull for maximal cocycles of surface groups~\cite{savini:surface} by extending Burger, Iozzi and Wienhard's \emph{tightness} to the wider setting of measurable cocycles.


\bibliographystyle{amsalpha}
\bibliography{biblionote}
\end{document}